\documentclass[a4paper, 11pt]{article}

\usepackage{amssymb}
\usepackage{amssymb}

\usepackage{amsmath,amssymb,amsthm}
\usepackage[latin1]{inputenc}
\usepackage{version,tabularx,multicol}
\usepackage{graphicx,float,psfrag}
\usepackage{stmaryrd}
\usepackage{color}
\usepackage{mathtools}
\mathtoolsset{showonlyrefs}
\usepackage[pdftex]{hyperref}
\usepackage{amsfonts}
\usepackage{mathrsfs}
\usepackage{geometry}
\usepackage{longtable}
\usepackage[numbers,sort&compress]{natbib}

\theoremstyle{plain}
\newtheorem{theorem}{Theorem}[section]

\newtheorem{corollary}[theorem]{Corollary}

\newtheorem{lemma}[theorem]{Lemma}

\newtheorem{remark}{Remark}[section]

\newcommand{\ind}{{\bf 1}}

\makeatletter
\@addtoreset{equation}{section}
\makeatother

\def\beqlb{\begin{eqnarray}}\def\eeqlb{\end{eqnarray}}
\def\beqnn{\begin{eqnarray*}}\def\eeqnn{\end{eqnarray*}}

\newcommand{\bcen}{\begin{center}}
	\newcommand{\ecen}{\end{center}}
\newcommand{\bgeqn}{\begin{equation}}
	\newcommand{\edeqn}{\end{equation}}

\begin{document}
	\title{On the maximal displacement of subcritical branching random walks	with or without	killing}
	\author{Haojie Hou, Shuxiong Zhang}
	\maketitle
	
	\renewcommand{\thefootnote}{\fnsymbol{footnote}}

	\begin{abstract}
		Consider a subcritical branching random walk $\{Z_k\}_{k\geq 0}$ with offspring distribution $\{p_k\}_{k\geq 0}$ and step size $X$.  Let $M_n$ denote the rightmost position reached by $\{Z_k\}_{k\geq 0}$ up to generation $n$,  and define $M := \sup_{n\geq 0} M_n$.
		In this paper we  give asymptotics of tail probability of  $M$ under optimal assumptions $\sum^{\infty}_{k=1}(k\log k) p_k<\infty$ and $\mathbb{E}[Xe^{\gamma X}]<\infty$, where $\gamma >0$ is a constant such that $\mathbb{E}[e^{\gamma X}]=\frac{1}{m}$ and  $m=\sum_{k=0}^\infty kp_k\in (0,1)$.  Moreover, we confirm the conjecture of Neuman and Zheng 
		[Probab. Theory Related Fields. 167 (2017) 1137--1164]
		by establishing the existence of a critical value $m\mathbb{E}[X e^{\gamma X}]$ such that
		\begin{align}
			\lim_{n\to\infty}e^{\gamma cn}\mathbb{P}(M_n\geq cn)=
			\begin{cases}
				&\kappa \in(0,1],~c\in\big(0,m\mathbb{E}[Xe^{\gamma X}]\big); \\
				&0,~~~~~~~~~~~c\in\big(m\mathbb{E}[Xe^{\gamma X}],\infty\big),
			\end{cases}
		\end{align}
		where $\kappa$ represents the non-zero limit.
		Finally, we extend these results to the maximal displacement of branching random walks with killing. Interestingly, this limit can be characterized through both the global minimum of a random walk with positive drift
		and the maximal displacement of the branching random walk without killing.
	\end{abstract}
	
	
	
	\section{Introduction and Main results}
	\subsection{Introduction}		
	The branching random walk (BRW) $\{Z_n\}_{n\geq 0}$ is a measure-valued Markov process whose law is determined by
	an offspring distribution $\{p_k\}_{k\geq 0}$ on $\mathbb{N}$ (the probability distribution of the number of children) and an $\mathbb{R}$-valued random variable $X$ representing the step size or displacement.
	There are several equivalent ways to define a branching random walk. In this work, we adopt the formulation presented in \cite{zheng17ptrf,lalley15ptrf}, which we now describe 
	formally.
	
	At generation $0$, there is one particle located at the origin (i.e. $Z_0=\delta_0$).  At generation $1$,
	this particle first performs a jump to some position $X_1 \stackrel{\mathrm{d}}{=} X$ and then produces $N$ offspring, where $N$ is equal in law to $\{p_k\}_{k\geq 0}$ and independent of $X_1$. Therefore, the point process $Z_1$ formed by all the
	particles
	 alive at generation $1$ is given by $Z_1= N\delta_{X_1}$.  At generation $2$, the particles alive at generation $1$ repeat their parent's behaviour independently from the place where they are born (first jump and then branch) and the procedure goes on.
	An alternative definition reverses the order of operations that each individual first branches and  then jumps. As noted in \cite[Remarks 2 and 3]{lalley15ptrf}, both formulations yield equivalent results for our purposes but
	the former one is more convenient to deal with.
	
	Let $\mathbb{T}$ denote the genealogical tree of the branching process
	rooted by $\emptyset$ with
	offspring distribution $\{p_k\}_{k\geq0}$; see \cite[p.13-14]{shi2015}.  Let
	$X_\emptyset=0$ and $\{X_v,\ v\in\mathbb{T}\setminus \{\emptyset \}\}$ be
	i.i.d. copies of $X$
	such that for each $v\in \mathbb{T}$, $X_v$ stands
	for the displacement of particle $v$. Let $V(\omega):=\sum_{v\preceq\omega}X_v$ be the position of particle $\omega\in\mathbb{T}$, where $v\preceq \omega$
	means that $v=\omega$ or $v$ is an ancestor of $\omega$. For each particle $\omega\in\mathbb{T}$, we use $|\omega|$ to denote the generation of $\omega$. The 
	population and their locations
	 at generation $n$ forms a point process given by
	$$Z_n=\sum_{|\omega|=n}\delta_{V(\omega)}.$$
	The branching random walk $\{Z_n\}_{n\geq0}$ is called a subcritical (critical or supercritical) branching random walk if $m:=\sum^{\infty}_{k=0}kp_k\in [0,1)$ ($m=1$ or $m>1$).  We refer to Shi \cite{shi2015} for a more detailed overview of branching random walks.
	Throughout this article, for a finite point measure $\mu$, we write $\mathbb{P}_{\mu}$ ($\mathbb{E}_{\mu}$) for the probability law (expectation operator)
	of the  branching random walk with initial value $Z_0=\mu$.
	For notational convenience, we write $\mathbb{P} = \mathbb{P}_{\delta_0}$.
	\par

	In this paper, we consider the tail probabilities of the maximal displacements in a subcritical branching random walk.
	Define the following extremal quantities:
	$$M_n:=\max_{|\omega|=i,0\leq i\leq n}V(\omega)~\text{for}~n\geq0,~M:=\max_{n\geq0}M_n,$$
	which represent  the maximal position up to generation $n$ and the global maximum respectively.
	For the subcritical case, 
	since the 
	 population dies out in finite time $\mathbb{P}$-almost surely, 
	 we have 
	  $M_n\leq M<\infty$. Under the assumptions that $m \in (0,1)$, $\sum_{k \geq 0} k^3 p_k < \infty$ and that the step size $X$ is a centered integer-valued random variable such that
	$\mathbb{E}[e^{\theta X}]<\infty$ for all $\theta>0$,
	Neuman and Zheng \cite[Theorem 1.2]{zheng17ptrf} established the following asymptotics:
	\begin{equation}\label{NZ-result'}
		0 < \liminf_{n \to \infty} e^{\gamma n} \mathbb{P}(M \geq n) \leq \limsup_{n \to \infty} e^{\gamma n} \mathbb{P}(M \geq n) \leq 1,
	\end{equation}
	where $\gamma > 0$ is the solution to the equation $\mathbb{E}[e^{\gamma X}] = \frac{1}{m}$. Moreover, under the additional assumptions that $\text{esssup}X<\infty$ and that $X$ is nearly right-continuous, Neuman and Zheng \cite[Theorem 1.2]{zheng17ptrf} proved that the following limit exists
	\begin{align}\label{NZ-result}
		\lim_{n\to\infty}e^{\gamma n}\mathbb{P}(M\geq n)=: \kappa \in (0, 1].
	\end{align}
	Some additional related works on subcritical branching L\'{e}vy processes will be introduced in Section~\ref{S1.3}. To the best of our knowledge, no further refinements of \eqref{NZ-result'} and \eqref{NZ-result}  exist in the literature for either subcritical branching random walks or subcritical	branching L\'{e}vy processes
	in the case where the step size has unbounded positive jumps with light-tailed.
	
	For $M_n$, Neuman and Zheng \cite[Theorem 1.6]{zheng17ptrf} discovered the following phase transition behavior. 
	They proved that there
	 exist constants $0 < \underline{\delta} \leq \overline{\delta} < \infty$ and the same constant $\kappa$
	in
	\eqref{NZ-result} such that
	\begin{itemize}
		\item[(i)] For $c\in(0,\underline{\delta})$,
		\begin{align}\label{4rr4tr3e1}
			\lim_{n\to\infty}e^{\gamma cn}\mathbb{P}(M_n\geq cn)=\kappa.
		\end{align}
		\item[(ii)] For $c\in(\overline{\delta},\infty)$,
		\begin{align}\label{4rr4tr3e2}
			\lim_{n\to\infty}e^{\gamma cn}\mathbb{P}(M_n\geq cn)=0.
		\end{align}
	\end{itemize}
	In \cite[Remark 1.9]{zheng17ptrf}, Neuman and Zheng conjectured the sharpness of this transition $\underline{\delta} = \overline{\delta}$, and they verified this conjecture for a special class of branching random walks.
	\par
	In this paper, we are also interested in the branching random walk with killing, whose definition is as follows.
	For each initial position $y > 0$, consider the BRW $(\{Z_n\}_{n\geq0}, \mathbb{P}_{\delta_y})$ starting from $Z_0 = \delta_y$. We add a killing mechanism that
	particles and their descendants are killed upon first entering $(-\infty, 0]$.
	The surviving 
	population and their locations 
	 at generation $n$ is given by
	\[
	Z_n^{(0,\infty)}:= \sum_{|\omega|=n} \ind_{\{ V(v)>0,	\forall v\preceq \omega \}}\delta_{V(\omega)}.
	\]
	Similarly, define
	\[
	M_n^{(0,\infty)}:=\max_{|\omega|=i,0\leq i\leq n, \forall v\preceq\omega , V(v)>0 }V(\omega)~\text{for}~n\geq0,~M^{(0,\infty)}:=\max_{n\geq0}M_n^{(0,\infty)}.
	\]
	
	In the continuous-time setting where the spatial motion is a spectrally negative L\'{e}vy process	and that the branching rate is  $\beta > 0$,  under the assumption $\sum_{k=1}^\infty k(\log k)p_k<\infty$, it was proved in
	\cite{houzhu25, zhu25} that  there exist   constants $\gamma_*, \kappa_*\in (0,\infty)$ such that for any $y>0$,
	\begin{align}\label{Zhu-result}
		\lim_{x\to+\infty} e^{\gamma_*  x}\mathbb{P}_{\delta_y}(M^{(0,\infty)}>x)= \kappa_*W^{(\beta(1-m))}(y),
	\end{align}
	where $W^{(\beta(1-m))}$ is the scale function associated with the spectrally negative L\'{e}vy 
	process
	 with index $\beta(1-m)$.
	\par
	Our works are inspired by the results of \cite{zheng17ptrf} in branching random walks and \cite{houzhu25, zhu25} in branching L\'{e}vy processes.
	The main contributions of this paper are divided into two parts:
	\begin{itemize}
		\item[(i)] In branching random walk,
		we establish the existence of the limit $\kappa$ in \eqref{NZ-result} under weaker assumptions (Theorem~\ref{456hy5t1}), requiring only that
		$\mathbb{E}[Xe^{\gamma X}] < \infty$ for the step size and that
		$\sum_{k=1}^\infty (k\log k) p_k < \infty$ for the offspring distribution. These conditions are shown to be optimal.  	
		We also confirm the conjecture
		of
		\cite[Remark 1.9]{zheng17ptrf}  in Theorem~\ref{54th6uyju5t} below and we prove  that
		\[
		\underline{\delta} = \overline{\delta} = m\mathbb{E}[Xe^{\gamma X}],
		\]
		which yields the sharp transition between \eqref{4rr4tr3e1} and \eqref{4rr4tr3e2}.
		\item[(ii)] For branching random walk with killing,
		we extend the spectrally negative L\'{e}vy process results of \eqref{Zhu-result} to discrete-time BRW with killing
		where  only $\mathbb{E}[Xe^{\gamma X}]<\infty$  and $\sum_{k=1}^\infty k(\log k)p_k<\infty$ are required.
		Compared to \cite{zhu25}, we allow the spatial motion to have positive jump.
	\end{itemize}
	\par

	\subsection{Main results}
	Recall that we only consider the subcritical case
	$m\in (0,1)$ (hence $p_0<1$).
	Moreover, in the sequel of this paper, we always assume that there exists a constant $\gamma>0$ such that
	$$\mathbb{E}[e^{\gamma X}]=\frac{1}{m}.$$
	\par
	Our first theorem concerns the tail probability of $M$, which shows that it decays exponentially. Let $u(x):=\mathbb{P}(M>x)$ for $x\in\mathbb{R}$. Set
	$\psi(s):=\sum^{\infty}_{n=0}p_n(1-s)^n+ms-1$ for $s\in[0,1]$, where we define $0^0=1$. Denote by $\{S_n\}_{n\geq0}$ the random walk with step size $X$. Let $\mathcal{F}_n:=\sigma(S_0,S_1,...,S_n),~n\geq0$ be the natural filtration of $\{S_n\}_{n\geq0}$. Let $\mathbb{P}_x$ ($\mathbb{E}_x$) be the probability measure (expectation) under which the random walk $\{S_n\}_{n\geq0}$ has initial value $x$, and (in the absence of confusion) write $\mathbb{P}=\mathbb{P}_{0}$ for short. Since $\mathbb{E}[e^{\gamma X}]=\frac{1}{m}$, $(\mathbb{P}_0,\{\mathcal{F}_n\}_{n\geq0},\{m^ne^{\gamma S_n}\}_{n\geq0})$ is a martingale with mean $1$. Thus, we can define the probability $\widehat{\mathbb{P}}_x$ as the following:
	\begin{align}\label{5hyth6y}
		\frac{\mathrm{d}\widehat{\mathbb{P}}_x}{\mathrm{d}\mathbb{P}_x}\Big{|}_{\mathcal{F}_n}:=m^n
		e^{\gamma (S_n-x)}
		,~n\geq0.
	\end{align}
	It is easy to check that under the probability $\widehat{\mathbb{P}}_x$, $\{S_n\}_{n\geq 0}$ is a random walk with initial value $x$. Thus, the above definition indeed makes sense. The same as before, without causing confusion, write $\widehat{\mathbb{P}}=\widehat{\mathbb{P}}_0$ for short. Define $\mathcal{T}_1:= \inf\{n\geq 0: S_n >S_0\}$ and $ \mathcal{H}_1:=S_{\mathcal{T}_1}$. We will prove in \eqref{step1}  that  $\widehat{\mathbb{E}}_0(S_1)>0$. Therefore,
	\begin{align}\label{Def-of-I}
		\widehat{\mathbb{P}}_0 \left(I_\infty:= \min_{j\geq 0} S_j \in (-\infty, 0]  \right)=1.
	\end{align}
	The distribution $\widehat{\mathbb{P}}_0\big(I_\infty \in \mathrm{d}y\big)$ has been studied in \cite[p.396, (2.7)]{feller1971}.  As usual, we call a random variable $X$ is lattice if there exist $a\in\mathbb{R}$ and $h>0$ such that $\mathbb{P}(X\in\{a+kh:k\in \mathbb{Z}\})=1$, where the largest such $h$ is called the span of $X$ and $a$ is called the shift. Otherwise, we say $X$ is non-lattice. To simplify the statement of the theorem, in the sequel of this paper, when $X$ is lattice we always assume the shift $a=0$.
	\begin{theorem}\label{456hy5t1}
		Suppose that $\mathbb{E}[Xe^{\gamma X}]<\infty$ and $\sum^{\infty}_{k=1}(k\log k)p_k<\infty$ .
		\begin{itemize}
			\item[(i)]  If the step size is non-lattice, then
			\begin{align}
				\lim_{x\to\infty}e^{\gamma x}\mathbb{P}(M>x)=
				(1-p_0)\frac{1-\widehat{\mathbb{E}}_0[e^{-\gamma\mathcal{H}_1}]}{m\widehat{\mathbb{E}}_0[\mathcal{H}_1]\gamma}
				-\frac{\widehat{\mathbb E}_0\left[\int^{\infty}_{-I_\infty}
					e^{\gamma z}\psi(u(z))\mathrm{d}z\right]}{m
					\widehat{\mathbb{E}}_0[S_1]}
				\in (0,1].
			\end{align}
			\item[(ii)]  If the step size is lattice with span $h>0$, then
			$$
			\lim_{n\to\infty}e^{\gamma hn}\mathbb{P}(M>hn)=
			\frac{(1-p_0)h(1-\widehat{\mathbb{E}}_0[e^{-\gamma \mathcal{H}_1 }])}{m\widehat{\mathbb{E}}_0[\mathcal{H}_1](e^{\gamma h}-1)}
			- \frac{h}{m\widehat{\mathbb{E}}_0[S_1]} \widehat{\mathbb E}_0\Bigg[ \sum_{i=-I_{\infty}/h}^\infty  \psi(u(ih))e^{\gamma ih}
			\Bigg]\in (0,1].
			$$
		\end{itemize}
	\end{theorem}
	\begin{remark}
		\begin{itemize}
			\item[(i)] 	
			Although  Neuman and Zheng considered the tail probability $\mathbb{P}(M\geq x)$ in \eqref{NZ-result} while our Theorem \ref{456hy5t1} considers a different probability $\mathbb{P}(M>x)$, one can obtain the asymptotic behavior for $\mathbb{P}(M\geq x)$ from Theorem \ref{456hy5t1}.
			Indeed,
			if the step size is non-lattice, then for any $\delta>0$,
			\begin{align}
				&\lim_{x\to\infty}e^{\gamma x}\mathbb{P}(M>x)\leq
				\lim_{x\to\infty}e^{\gamma x}\mathbb{P}(M\geq x) \leq \lim_{x\to\infty}e^{\gamma x}\mathbb{P}(M>x-\delta )=e^{\gamma \delta}\lim_{x\to\infty}e^{\gamma x}\mathbb{P}(M> x).
			\end{align}
			Therefore, the first result of Theorem \ref{456hy5t1} remains true with $\mathbb{P}(M>x)$ replaced by $\mathbb{P}(M\geq x)$. If the step size is lattice, then naturally $\mathbb{P}(M>hn)=\mathbb{P}(M\geq h(n+1))$.
			\item[(ii)]  If $\mathbb{P}(X=1)=1$, then $\mathbb{P}(M>n)=
			\mathbb{P}(Z_{n}(\mathbb{R})>0)$.
			By \cite[p.40]{athreya72},   $\lim_{n\to\infty}\frac{1}{m^n}
			\mathbb{P}(Z_{n}(\mathbb{R})>0)
			\in(0,\infty)$ if and only if $\sum_{k=1}^\infty k(\log k)p_k<\infty$. Therefore, our assumption for the offspring law is optimal in this case. On the other hand, if $\mathbb{E}[Xe^{\gamma X}]=\infty$, then $\widehat{\mathbb{E}}_0[\mathcal{H}_1]=\widehat{\mathbb{E}}_0[S_1]=\infty$; see \cite[p.397, Theorem 2 (iii)]{feller1971}. Thus, limits in Theorem \ref{456hy5t1} equals to $0$, which means our assumption for the step size is optimal.
		\end{itemize}
	\end{remark}
	
	Our next result considers asymptotic behaviors of $M_n$ condition on $M>cn$, which exhibits a phase transition at $c=m\mathbb{E}[Xe^{\gamma X}]$.
	
	\begin{theorem}\label{54th6uyju5t}
		Assume
		$\mathbb{E}[Xe^{\gamma X}]<\infty$ and $\sum^{\infty}_{k=1}(k\log k)p_k<\infty$.
		\begin{itemize}
			\item[(i)] 	If $c\in\big(0,m\mathbb{E}[Xe^{\gamma X}]\big)$, then
			\begin{align}
				\lim_{n\to\infty}\mathbb{P}(M_n\geq cn \big| M\geq cn)=1.
			\end{align}
			\item[(ii)] 	If $c\in(m\mathbb{E}[Xe^{\gamma X}],\infty)$, then
			\begin{align}
				\lim_{n\to\infty}\mathbb{P}(M_n\geq cn \big| M\geq cn) =0.
			\end{align}
		\end{itemize}
	\end{theorem}
	\begin{remark}
		If $X$ is non-lattice or lattice with span $1$, then
		\begin{align}
			\lim_{n\to\infty}e^{\gamma cn}\mathbb{P}(M_n\geq cn)
			&=\lim_{n\to\infty}e^{\gamma cn}{\mathbb{P}(M\geq cn)}\mathbb{P}(M_n\geq cn \big| M\geq cn)\cr
			&=
			\begin{cases}
				&\text{exists}\in(0,1],~c\in\big(0,m\mathbb{E}[Xe^{\gamma X}]\big); \\
				&0,~~~~~~~~~~~~~~~~c\in\big(m\mathbb{E}[Xe^{\gamma X}],\infty\big),
			\end{cases}
		\end{align}
		where the first equality follows from the fact $M_n\leq M$, and the second equality follows from Theorem \ref{456hy5t1} and Theorem \ref{54th6uyju5t}. Thus, we have confirmed and generalized the conjecture of  Neuman and Zheng \cite[Remark 1.9]{zheng17ptrf}.
	\end{remark}
	
	Now we are ready to state our main results for the branching random walk with killing. If $X$ is lattice with span $h$, then observe that $\mathbb{P}_{\delta_y}\left(M^{(0,\infty)}> kh  \right) =\mathbb{P}_{\delta_{ih}}\left(M^{(0,\infty)}> (k-1)h  \right)$, where $y\in (ih, (i+1)h)$ for some $i\geq1$. Therefore, in the following results, we always assume that $y=ih$ for some $i\geq1$ whenever $X$ is lattice with span $h$.
	
	\begin{theorem}\label{theom3}
		Assume
		$\mathbb{E}[Xe^{\gamma X}]<\infty$ and $\sum^{\infty}_{k=1}(k\log k)p_k<\infty$ .
		\begin{itemize}
			\item[(i)] If the step size is non-lattice, then for any $y>0$,
			\begin{align}
				\lim_{x\to\infty} \frac{\mathbb{P}_{\delta_y}(M^{(0,\infty)}>x)  }{\mathbb{P}(M>x)} =e^{\gamma y} \widehat{\mathbb{P}}_0(I_\infty>-y) .
			\end{align}
			\item[(ii)] If the step size is lattice with span $h>0$,	then for any integer $i\geq1$,
			$$
			\lim_{n\to\infty}
			\frac{\mathbb{P}_{\delta_{ih}}(M^{(0,\infty)}>hn) }{\mathbb{P}(M>hn)}=	e^{\gamma ih} \widehat{\mathbb{P}}_0(I_\infty>-ih).
			$$
		\end{itemize}
	\end{theorem}
	
	Finally, the last theorem considers the long time behaviour of the maximum
	from
	time $0$ to $n$ of the branching random walk with killing, whose results are almost the same as those of the branching random walk without killing.
	\begin{theorem}\label{theom4}
		Assume
		$\mathbb{E}[Xe^{\gamma X}]<\infty$ and $\sum^{\infty}_{k=1}(k\log k)p_k<\infty$. Then, the results of Theorem \ref{54th6uyju5t} hold with $(M_n, M, \mathbb{P})$ replaced by $(M_n^{(0,\infty)}, M^{(0,\infty)},
		\mathbb{P}_{\delta_{y}})$, where $y$ is any positive constant if $X$ is non-lattice and $y=ih$ for some $i\geq1$ if $X$ is lattice with span $h$.
	\end{theorem}

	\subsection{Discussion on the literature}\label{S1.3}
	
	The tail asymptotics of maximal displacements in branching random walks and branching L\'{e}vy processes have been extensively studied. The earliest foundational work dates back to Sawyer and Fleischman \cite{sawyer1979}, who employed analytic methods to investigate the maximal displacement $M$ in critical and subcritical branching Brownian motions under a third moment condition $\sum_{k=0}^\infty k^3p_k<\infty$ on the offspring distribution.
	
	For branching random walks, the analytic techniques in \cite{sawyer1979} are no longer applicable. In the critical case, significant progress was made by
	Lalley and Shao \cite{lalley15ptrf} and Zhang \cite{zhang25} who utilized discrete Feynman-Kac formulas to establish that when the step size $X$ admits finite fourth moment $\mathbb{E}[X^4] < \infty$ and that  $\sum_{k=0}^\infty k^2 p_k < \infty$,
	the tail probability exhibits polynomial decay:
	\begin{equation}\label{eq:critical_tail}
		\lim_{x \to \infty} x^2 \mathbb{P}(M > x) = \frac{6 \text{Var}(X)}{\text{Var}(Z_1(\mathbb{R}))}.
	\end{equation}
	The asymptotic properties of $M_n$ have been studied in various regimes by Kesten \cite{Kesten1995} in critical case
	and by Neuman and Zheng \cite{zheng21ptrf} in near-critical case.
	
	Analogous results in critical and subcritical branching L\'{e}vy processes are developed. Profeta \cite{profeta24bernoulli}
	employed scale functions to derive precise tail asymptotics for both subcritical and critical branching spectrally negative L\'{e}vy processes.
	For the case where $\sum_{k=0}^\infty k^2p_k=\infty$, Hou et al.  \cite{hou25} and Profeta \cite{Profeta2025} studied the critical branching L\'evy process with stable offspring law.
	\par
	
	We now review some literatures addressing cases where the spatial motion lacks finite second 
	moment.
	For critical branching $\alpha$-stable processes with $\alpha \in (0,2)$,
	Lalley and Shao \cite{lalley16} first studied the symmetric $\alpha$-stable case, proving the polynomial tail decay:
	\begin{equation}\label{eq:stable-tail}
		\lim_{x\to\infty} x^{\alpha/2}\mathbb{P}(M > x) = \sqrt{\frac{2}{\alpha}}.
	\end{equation}
	Profeta \cite{profeta22alea,Profeta2025} extended \eqref{eq:stable-tail} to
	critical and subcritical branching L\'evy processes with positive heavy-tailed jumps.  For critical and subcritical branching spectrally negative $\alpha$-stable processes, see  Hou et al. \cite{HJRS25}.
	\par
	
	In the supercritical regime, since $M = \infty$ almost surely on the survival set, research mainly focuses on the asymtotic behavior of $M_n$.
	In super-Brownian motion, Pinsky \cite{pinsky95}
	(although the definition of $M_t$ is slightly different from ours, one can compare \cite[Corollary 3.2]{KLMR} for the lower bound)
	proved that $M_t$ grows linearly with $t$. Later, Zhang \cite{zhangecp2024} obtained decay rates for $\mathbb{P}(M_t \geq ct)$ when $c$ exceeds a critical threshold.
	Similar large deviation results were proved by \.{O}z et al. \cite{mehet17} in branching Brownian motion.
	\par
	
	For branching random walks and branching L\'{e}vy processes with killing, as previously mentioned, results in the subcritical case $m\in (0,1)$ were established by \cite{houzhu25,zhu25}. In the critical case $m=1$, Lalley and Zheng \cite{lalley-zheng15} first studied the Brownian motion case, determining the tail probability of $M^{(0,\infty)}$ using analytical methods. Later, the results of \cite{lalley-zheng15} was generalized by Hou et. al. \cite{hou25+}  through probabilistic techniques. In supercritical case $m>1$, under the condition that the step size has negative drift (ensuring finite-time extinction),
	A\"{i}d\'{e}kon et al. \cite{AHO13} characterized the tail behavior of $M^{(0,\infty)}$.  In the continuous-time setting, we refer to \cite{MS25+} and the references therein.

	\subsection{Proof strategies of the main results}
	
	Our approach differs fundamentally from previous methods employing Feynman-Kac formulas
	and Fekete's subadditive lemma in \cite{zheng17ptrf,houzhu25,zhu25}. The key components of our proof are as follows.
	
	In the proof of Theorem \ref{456hy5t1}, firstly we utilize height random variables and the
    renewal theory to  show that the Laplace transform of the overshoot $(S_{T^+_0}, \widehat{\mathbb{P}}_{-x})$ converges as $x\to+\infty$, where $T_0^+:=\inf\{n\in \mathbb{N}: S_n>0\}$; see Lemma  \ref{lem1} below. We also prove another renewal theorem of $(\{S_n\}_{n\geq 0}, \widehat{\mathbb{P}}_{-x})$ 
    and we show the existence of the limit $\sum_{\ell=1}^\infty  \widehat{\mathbb{E}}_{-x}\left(f(S_\ell)\ind_{\{\max_{0\leq j\leq \ell} {S}_j \leq 0\}} \right)$ as $x\to +\infty$  in Lemma  \ref{lem1}.
	Then we express $\mathbb{P}(M > x)$
	as functionals
	of $(\{S_n\}_{n\geq0}, \widehat{\mathbb{P}}_x)$. Finally, we analyze these functionals directly according to renewal 
	theorems
	 and prove that the
	limit $\lim_{x\to \infty} e^{\gamma x}\mathbb{P}(M>x)$
	exists. The Feynman-Kac formula is only used in the proof for the positivity of the  limit.
	
	In the proof of Theorem \ref{54th6uyju5t}, if $c \in (0, m\mathbb{E}[Xe^{\gamma X}])$, then we apply the many-to-one formula (see \eqref{many-to-one-formula} below) and Markov property to
	show the identical asymptotics for $\mathbb{P}(M_n \geq cn)$ and $\mathbb{P}(M \geq cn)$. For $c \in (m\mathbb{E}[Xe^{\gamma X}], \infty)$, we construct a  submartingale
	$\{D_n^+\}_{n\geq0}$ (see \eqref{Def-of-D-n} below).
	Then from $\{M_n\geq cn\}\subset \{ \max_{s\leq n} D_s^+\geq (c-\widehat{\mathbb{E}}_0[S_1])ne^{\gamma cn}\}$ and Doob's maximal inequality, we find $e^{\gamma cn}\mathbb{P}(M_n\geq cn)$ can be bounded above by the moment of $D_n^+$. Now  the ergodic theorem yields that  $\lim_{n\to+\infty} e^{\gamma cn}\mathbb{P}(M_n\geq cn)=0$. 
	
	In the proofs of Theorems \ref{theom3} and \ref{theom4} for analogous results of branching random walk with killing, the proof strategy for the tail probability of $M$ is applicable to $M^{(0,\infty)}$, so  we  follow a similar way in the proof of Theorem \ref{theom3}. The proof of  Theorem \ref{theom4} is similar to that of Theorem \ref{54th6uyju5t}.

	\section{Properties for random walk with positive drift}
	
	Let $U(\mathrm{d}y):=\sum^{\infty}_{n=0}\widehat{\mathbb{P}}_0(S_n\in \mathrm{d}y)$. Denote by $U(y):=U((-\infty,y]),~y\in\mathbb{R}$ the renewal function of $\{S_n\}_{n\geq0}$. By the Markov inequality, for any $y\in \mathbb{R}$,
	\begin{align}\label{45gtzy6uy6}
		U(y)&=\sum^{\infty}_{n=0}\widehat{\mathbb{P}}_0(S_n\leq y)
		\leq \sum^{\infty}_{n=0}\widehat{\mathbb{E}}_0[e^{\gamma(y-S_n)}]\cr
		&=e^{\gamma y}\sum^{\infty}_{n=0}(\widehat{\mathbb{E}}_0[e^{-\gamma X}])^n
		=e^{\gamma y}\sum^{\infty}_{n=0}m^n<\infty,
	\end{align}
	where the last equality follows from (\ref{5hyth6y}). Thus, $U(y)$ is well-defined. The following lemma presents several properties of the renewal function $U$.
	\begin{lemma}\label{4rtgt5gt}
		Assume
		$\mathbb{E}[Xe^{\gamma X}]<\infty$.
		\begin{itemize}
			\item[(i)] $U(x+y)\leq U(x)+U(y)$ for any $x,y\in\mathbb{R}$.
			
			\item[(ii)] $U(x)\leq (x+1)U(1)$ for all $x>1$.
		\end{itemize}
	\end{lemma}
	\begin{proof}
		(i) Since $U$ is a non-negative increasing function in $\mathbb{R}$, we only prove the case $x,y>0$.
		Noticing that $g(u):= u\log u$ is a convex function on $(0,\infty)$.
		Thus, by Jensen's inequality, it follows from \eqref{5hyth6y} that
		\begin{align}\label{step1}
			\widehat{\mathbb{E}}_0(S_1) = m \mathbb{E}(Xe^{\gamma X})= \frac{m}{\gamma} \mathbb{E}\left(g(e^{\gamma X})\right)\geq  \frac{m}{\gamma} g\left(\mathbb{E}(e^{\gamma X})\right) = \frac{1}{\gamma}\log \frac{1}{m}>0.
		\end{align}
		Therefore,  by the strong law of large numbers, we have $\lim_{n\to\infty}S_n=+\infty$, $\widehat{\mathbb{P}}_0$-almost surely, which implies that  for any $x>0$, $\widehat{\mathbb{P}}_0$-almost surely,
		\begin{align}\label{Def-of-T+}
			T_x^+:=\min\{n\geq0: S_n>x\}<\infty.
		\end{align}
		According to the definition of $T_x^+$, for every $x, y>0$, it holds that
		\begin{align}\label{4gt5ty6}
			\sum^{T_x^+-1}_{n=0}\ind_{\{S_n\leq x+y\}}= \sum^{T_x^+-1}_{n=0}\ind_{\{S_n\leq x\}}\leq \sum^{\infty}_{n=0}\ind_{\{S_n\leq x\}}.
		\end{align}
		Noticing that on $\{S_n \leq x+y\}$, we have $S_{T_x^+} >x$. Therefore,
		\begin{align}
			S_n-S_{T_x^+}\leq x+y-x\leq y,
		\end{align}
		which implies that
		\begin{align}\label{erf4ghy4}
			\sum^{\infty}_{n=T_x}\ind_{\{S_n\leq x+y\}}
			\leq\sum^{\infty}_{n=T_x}\ind_{\{S_n-S_{T_x}\leq y\}}
			\leq\sum^{\infty}_{k=0}\ind_{\{S_{k+T_x}-S_{T_x}\leq y\}}.
		\end{align}
		Combining \eqref{4gt5ty6} and \eqref{erf4ghy4}, we obtain that
		\begin{align}\label{45thyh6j}
			\sum^{\infty}_{n=0}\ind_{\{S_n\leq x+y\}}\leq \sum^{\infty}_{n=0}\ind_{\{S_n\leq x\}}+\sum^{\infty}_{k=0}\ind_{\{S_{k+T_x}-S_{T_x}\leq y\}}.
		\end{align}
		By the strong Markov property,
		$\{S_{k+T_x}-S_{T_x}\}_{k\geq0}$
		is a random walk with the same distribution as $\{S_k\}_{k\geq0}$. Thus, taking expectation on both sides of \eqref{45thyh6j} yields that
		$$
		U(x+y)\leq U(x)+U(y).
		$$
		This completes the proof of (i).

		(ii)
		Denote by $\lfloor x\rfloor$ the greatest integer less than or equal to $x$. It follows from (i) that  for any $x\geq1$,
		\begin{align}
			U(x)=U(\lfloor x\rfloor+x-\lfloor x\rfloor)\leq U(\lfloor x\rfloor)+U(x-\lfloor x\rfloor)\leq\lfloor x\rfloor U(1)+U(1)\leq(x+1)U(1),\nonumber
		\end{align}
		where in the second inequality we  used the fact that $U$ is increasing on $\mathbb{R}$.
	\end{proof}
	
	The following lemma is the classical renewal theorem, which can be found in \cite[Theorem 3.1, p.169 and Proposition 3.4, p.171]{Revuz1984}. For $n\geq0$, put $S^Y_n:=Y_0+...+Y_n$, where $Y_i,i\geq1$ are i.i.d. and $Y_0$ is the initial value. Let $\mathbf{P}_x$ ($\mathbf{E}_x$) be the probability measure (expectation) such that the random walk $\{S^Y_n\}_{n\geq0}$ has initial value $x$. Put
	\[
	U^Y(\mathrm{d} y):=  \sum^{\infty}_{n=0}\mathbf{P}_{-x}(S_n^Y \in \mathrm{d} y),~x>0.
	\]
	
	\begin{lemma}\label{Lem-renewal-theorem} Assume $\mathbf{E}_0(Y_1)>0$.
		
		\begin{itemize}
			\item[(i)] If $(Y_1,\mathbb{P}_0)$ is non-lattice, then
			$U^Y(\mathrm{d} y)$ converges vaguely to $\frac{1}{\mathbf{E}_0[Y_1]}\mathrm{d} y$ as $x\to+\infty$. In particular, for any $y>0$ and $c\in\mathbb{R}$, $$\lim_{x\to+\infty} U^Y((c,c+y])=\frac{1}{\mathbf{E}_0[Y_1]}y.$$
			\item[(ii)] If $(Y_1,\mathbb{P}_0)$ is lattice with span $h$, then
			for any $i\in \mathbb{Z}$,
			\begin{align}
				\lim_{k\to+\infty} \sum_{n=0}^\infty \mathbf{P}_{-kh}(S_n^Y= ih) = \frac{h}{\mathbf{E}_0[Y_1]}.
			\end{align}
		\end{itemize}
	\end{lemma}
	
	For each $n\geq 0$, define
	\begin{align}\label{Def-of-max-S}
		\overline{S}_n:= \max_{0\leq j\leq n}S_j\quad\mbox{and}\quad \underline{S}_n:= \min_{0\leq j\leq n}S_j.
	\end{align}
	Different from the classical renewal theorem in Lemma \ref{Lem-renewal-theorem}, our next lemma presents a renewal theorem associated with  the maximum of the random walk $\{S_j\}_{j\geq0}$ up to time $n$.

	\begin{lemma}\label{lem2}
		Assume  $\mathbb{E}[Xe^{\gamma X}]<\infty$.
		
		\begin{itemize}
			\item[(i)] 	If $X$ is non-lattice, then for any 	$b<a\leq 0$,
			\begin{align}
				\lim_{x\to+\infty}\sum^{\infty}_{n=1}\widehat{\mathbb{P}}_{-x}\Big(\overline{S}_n \leq 0,
				S_n\in (b, a] \Big)
				&=\frac{1}{\widehat{\mathbb{E}}_0[S_1]}\int_{(b, a]}   \widehat{\mathbb{P}}_0 \left(I_\infty \in [z,0]\right) \mathrm{d} z,
			\end{align}
			where $I_\infty$ is given as in \eqref{Def-of-I}.
			\item[(ii)] If $X$ is lattice with span $h$, then for any	$i\in\mathbb{Z}, i\leq 0$,
			\begin{align}
				\lim_{k\to+\infty}\sum^{\infty}_{n=1}\widehat{\mathbb{P}}_{-kh}\Big(\overline{S}_n \leq 0,
				S_n= ih \Big)
				&= \frac{h}{\widehat{\mathbb{E}}_0[S_1]}\widehat{\mathbb{P}}_0(I_\infty\in[ih, 0]).
			\end{align}
		\end{itemize}
		
	\end{lemma}
	
	\begin{proof}
		(i)
		Since
		$$\lim_{x\to+\infty}\widehat{\mathbb{P}}_{-x}\Big(\overline{S}_n \leq 0,S_n\in (b, a] \Big)=0$$ for $n=0$,
		it suffices to prove that
		\begin{align}
			\lim_{x\to+\infty}\sum^{\infty}_{n=0}\widehat{\mathbb{P}}_{-x}\Big(\overline{S}_n \leq 0,
			S_n\in (b, a] \Big)
			&=\frac{1}{\widehat{\mathbb{E}}_0[S_1]}\int_{(b, a]}   \widehat{\mathbb{P}}_0 \left(I_\infty \in [z,0]\right) \mathrm{d} z.
		\end{align}
		For $n\geq 0$, define
		$$\tau_n=\min\{0\leq j\leq n:S_j=\overline{S}_n\}.$$
		Define $S_{k,j}:= S_{k+j}-S_j$ and $\overline{S}_{k,j}:= \max_{0\leq q\leq k} S_{q,j}$ for $k,j\geq0.$ Then $\{S_{k,j}: k\geq 0\}$ is a random walk independent of $\sigma(S_1,...,S_j)$. Therefore,
		\begin{align}\label{4rtw6hy}
			&\sum^{\infty}_{n=0}\widehat{\mathbb{P}}_{-x}( \overline{S}_n \leq 0, S_n\in (b,a] )
			=
			\sum^{\infty}_{n=0}\sum^n_{j=0}\widehat{\mathbb{P}}_{-x}(\tau_n=j,S_j\leq 0,S_n\in (b,a])\cr
			&=\sum^{\infty}_{n=0}\sum^n_{j=0}\widehat{\mathbb{P}}_{-x}(S_k<S_j,\forall k\in\{0,...,j-1\},S_j\leq 0,S_k\leq S_j,\forall k\in\{j,...,n\},  S_n\in (b,a])\cr
			&=\sum^{\infty}_{j=0}\sum^{\infty}_{n=j}\widehat{\mathbb{P}}_{-x}(\overline{S}_{j-1}<S_j,S_j\leq 0,\overline{S}_{n-j,j}\leq 0, S_{n-j,j}+S_j\in (b,a])\cr
			&=\sum^{\infty}_{j=0}\sum^{\infty}_{i=0}\widehat{\mathbb{P}}_{-x}(\overline{S}_{j-1}<S_j,S_j\leq 0, \overline{S}_{i,j} \leq 0,  S_{i,j}+S_j\in (b,a] ).
		\end{align}
		Put
		\begin{align}\label{46fvrfhvfr}
			\mathcal{U}(\mathrm{d}y)&:=\sum^{\infty}_{i=0}\widehat{\mathbb{P}}_0\big(\overline{S}_i\leq0,S_i\in \mathrm{d}y\big)~\text{on}~(-\infty, 0],\cr
			V(\mathrm{d}z)&:=\sum^{\infty}_{j=0}\widehat{\mathbb{P}}_0\big(\min_{1\leq i\leq j}S_i>0 ,{{S}}_j\in \mathrm{d}z\big)~\text{on}~[0,\infty),
		\end{align}
		where we define $\min_{1\leq i\leq0}S_i=+\infty$ by convention.
		Also, we define $V(z):= V([0,z])$ for $z\geq 0$.
		Let
		$$S^{(j)}_k:=S_j-S_{j-k}, \ 0\leq k\leq j$$
		be the backward random walk of $\{S_k\}_{0\leq k\leq j}$. Combining  \eqref{4rtw6hy} and  \eqref{46fvrfhvfr},
		when $x+b>0$, we have
		\begin{align}\label{4yhyuo9q1}
			&\sum^{\infty}_{n=0}\widehat{\mathbb{P}}_{-x}( \overline{S}_n \leq 0, S_n\in (b,a] )\cr
			&=\int_{(-\infty, 0]}\sum^{\infty}_{j=0}\widehat{\mathbb{P}}_0\Big(\overline{S}_{j-1}< S_j, S_j\leq x, S_j\in(x+b-y, x+a-y]\Big)\mathcal{U}(\mathrm{d}y)\cr
			&=\int_{(-\infty, 0]}\sum^{\infty}_{j=0}\widehat{\mathbb{P}}_0\Big(S^{(j)}_{j-k}>0,\forall k\in\{0,...,j-1\}, x\geq S^{(j)}_j,{S}^{(j)}_j\in(x+b-y,x+a-y]\Big)\mathcal{U}(\mathrm{d}y)\cr
			&=\int_{(-\infty, 0]}\sum^{\infty}_{j=0}\widehat{\mathbb{P}}_0\big(\min_{1\leq i\leq j}S_i>0, x\geq S_j,{S}_j\in(x+b-y,x+a-y]\big)\mathcal{U}(\mathrm{d}y)\cr
			&=\int_{(a,0]}V(x+a-y)-V(x+b-y)\mathcal{U}(\mathrm{d}y)+\int_{(b, a]}V(x)-V(x+b-y)\mathcal{U}(\mathrm{d}y),
		\end{align}
		where the last equality follows from the fact
		that $x+b-y>0$ for all $y\leq 0$ and
		for $y\leq b$,
		$\{x\geq S_j\}\cap \{{S}_j\in(x+b-y,x+a-y]\}=\emptyset.$
		\par
		Next, we are going to show that the measure $\mathcal{U}$ can be characterized through the distribution function of $I_{\infty}$. Since $\widehat{\mathbb{P}}_0(I_\infty \in (-\infty, 0])=1$, we have $\widehat{\mathbb{P}}_0$-almost surely,
		$$\tau:=\max\{n\geq0: S_n=I_\infty \}<\infty.$$
		Moreover, we have
		$$\rho:=\widehat{\mathbb{P}}_0\left(S_i>0,\forall i\geq 1\right)\in	(0,1].$$
		Indeed,  since $\widehat{\mathbb{P}}_0(I_\infty >-\infty)=1$, there exists $K>0$ such that
		\begin{align}\label{step29}
			\widehat{\mathbb{P}}_0(I_\infty>-K)>0.
		\end{align}
		Now from $\widehat{\mathbb{P}}(X>0)>0$, we see that there exists $\varepsilon>0$ such that $\widehat{\mathbb{P}}_0(S_1>\varepsilon)>0$.  Hence, by \eqref{step29}, set $n_0:=\lfloor K/\varepsilon \rfloor +1$,
		it holds that
		\begin{align}
			\rho&\geq \widehat{\mathbb{P}}_0\left(S_j- S_{j-1}> \varepsilon, \forall 1\leq j\leq n_0,\  S_{j+n_0}-S_{n_0}>-K,\ \forall j\geq 0 \right)\cr
			&\geq  (\widehat{\mathbb{P}}_0(S_1>\varepsilon))^{n_0}	\widehat{\mathbb{P}}_0(I_\infty>-K)>0.
		\end{align}
		Observe that for any $z<0$,
		\begin{align}\label{54trh62w}
			& \widehat{\mathbb{P}}_0\left(I_\infty \geq z\right)
			=\sum^{\infty}_{n=0}\widehat{\mathbb{P}}_0\left(S_n\geq z,\tau=n\right)=\sum^{\infty}_{n=0}\widehat{\mathbb{P}}_0\left(\underline{S}_{n-1}\geq S_n\geq z,  S_k> S_n,\forall k\geq n+1 \right)\cr
			&=\widehat{\mathbb{P}}_0\left(S_i>0,\forall i\geq 1\right)\sum^{\infty}_{n=0}\widehat{\mathbb{P}}_0\left(\underline{S}_{n-1}\geq S_n\geq z \right)=\rho \sum^{\infty}_{n=0}\widehat{\mathbb{P}}_0\left( \overline{S}_n^{(n)}\leq 0,   \overline{S}_n^{(n)}\geq z \right) \cr
			&=\rho \sum^{\infty}_{n=0}\widehat{\mathbb{P}}_0\left( \overline{S}_n\leq0,S_n\geq z\right)=\rho \mathcal{U}([z,0])
			,
		\end{align}
		where in the second inequality we define $S_{-1}:=+\infty$ for convenience.
		\par
		Let $q>0$ be a constant and $p(x), x\geq0$ be a positive function such that  $\lim_{x\to\infty}p(x)=+\infty.$
		We are going to prove that
		\begin{align}\label{4ty7us2}
			\lim_{x\to+\infty}
			\big( V(p(x)+q)-V(p(x)) \big)
			=\frac{\rho  q}{\widehat{\mathbb{E}}_0[S_1]}.
		\end{align}
		For simplicity, for $N\geq1$ and $x\geq0$, we define
		\begin{align}
			J_{1, N}(x)& := \bigg|\sum^{\infty}_{j=N+1}\widehat{\mathbb{P}}_0\left(\min_{1\leq i\leq j}{S}_i>0,{{S}}_j\in (p(x),p(x)+q]\right)\nonumber\\ &\qquad -\sum^{\infty}_{j=N+1}\widehat{\mathbb{P}}_0\left(\min_{1\leq i\leq N}S_i>0,{{S}}_j\in (p(x),p(x)+q]\right)\bigg|,\nonumber\\
			J_{2,N}(x)&:= \left|\sum^{\infty}_{j=N+1}\widehat{\mathbb{P}}_0\left(\min_{1\leq i\leq N}S_i>0,{{S}}_j\in (p(x),p(x)+q]\right)-\frac{q}{\widehat{\mathbb{E}}_0[S_1]}{\widehat{\mathbb{P}}_0\Big(\min_{1\leq i\leq N}S_i>0\Big)}\right|.
		\end{align}
		Through simple observations, it follows that for any $N\geq1$,
		\begin{align}\label{4grtgt6y6}
			&\left|V(p(x)+q)-V(p(x))-\frac{\rho  q}{\widehat{\mathbb{E}}_0[S_1]}\right|
			=\left|\sum^{\infty}_{j=0}\widehat{\mathbb{P}}_0\Big(\min_{1\leq i\leq j}{S}_i>0,{{S}}_j\in (p(x),p(x)+q]\Big)-\frac{\rho  q}{\widehat{\mathbb{E}}_0[S_1]}\right|\cr
			&\leq\sum^{N}_{j=0}\widehat{\mathbb{P}}_0\left({{S}}_j> p(x)\right)
			+J_{1,N}(x)+J_{2,N}(x)
			+\left|\frac{q}{\widehat{\mathbb{E}}_0[S_1]}{\widehat{\mathbb{P}}_0\Big(\min_{1\leq i\leq N}S_i>0\Big)}-\frac{\rho  q}{\widehat{\mathbb{E}}_0[S_1]}\right|.
		\end{align}
		We first deal with	$J_{1,N}(x)$.
		By the Markov property, we have
		\begin{align}
			J_{1,N}(x)
			&\leq \sum^{\infty}_{j=N+1}\widehat{\mathbb{P}}_0\left(\exists i\in{N+1,...,j}~\text{s.t.}~S_i\leq0,{{S}}_j\in (p(x),p(x)+q]\right)\cr
			&\leq\sum^{\infty}_{j=N+1}\sum^j_{i=N+1}\widehat{\mathbb{P}}_0\left(S_i\leq0,{{S}}_j\in (p(x),p(x)+q]\right)\cr
			&=\sum^{\infty}_{i=N+1}\sum^{\infty}_{j=i}\widehat{\mathbb{E}}_0\left[\ind_{\{S_i\leq0\}}
			\widehat{\mathbb{P}}_{z}\left(S_{j-i}\in(p(x),p(x)+q]\right)\Big|_{z=S_i}\right]\cr
			&=\sum^{\infty}_{i=N+1}\widehat{\mathbb{E}}_0\left[\ind_{\{S_i\leq0\}}
			\sum^{\infty}_{k=0}\widehat{\mathbb{P}}_{z}\left(S_{k}\in(p(x),p(x)+q]\right)\Big|_{z=S_i}\right]	.
		\end{align}
		From Lemma \ref{4rtgt5gt} (i), the above inequality is bounded from above by
		\begin{align}\label{step44}
			J_{1,N}(x)&\leq\sum^{\infty}_{i=N+1}\widehat{\mathbb{E}}_0\left[\ind_{\{S_i\leq0\}}\Big(U(p(x)+q-S_i)-U(p(x)-S_i)\Big)\right]\cr
			&\leq U(q)\sum^{\infty}_{i=N+1}\widehat{\mathbb{P}}_0\left(S_i\leq0\right).
		\end{align}
		We proceed to deal with $J_{2,N}(x)$. Noticing that  by the Markov property,
		\begin{align}
			&\lim_{x\to\infty}\sum^{\infty}_{j=N+1}\widehat{\mathbb{P}}_0\left(\min_{1\leq i\leq N}S_i>0,{{S}}_j\in (p(x),p(x)+q]\right)\cr
			&=\lim_{x\to\infty}\widehat{\mathbb{E}}_0\left[\ind_{\{\min_{1\leq i\leq N}S_i>0\}}\sum^{\infty}_{j=1}	\widehat{\mathbb{P}}_{z}\left(S_j\in (p(x),p(x)+q]\right)\Big|_{z=S_N}\right]\cr
			&=\widehat{\mathbb{E}}_0\left[\ind_{\{\min_{1\leq i\leq N}S_i>0\}}\lim_{x\to\infty}\sum^{\infty}_{i=1}
			\widehat{\mathbb{P}}_{z}\left(S_i\in(0,q]\right)|_{z=S_N-p(x)}\right]\cr
			&=\frac{q}{\widehat{\mathbb{E}}_0[S_1]}\widehat{\mathbb{P}}_0\Big(\min_{1\leq i\leq N}S_i>0\Big),
		\end{align}
		where the last equality follows from Lemma \ref{Lem-renewal-theorem} and the second equality follows from the dominated convergence theorem and the fact
		that for all $z\in \mathbb{R}$,
		\begin{align}
			\sum^{\infty}_{i=1}\widehat{\mathbb{P}}_{z}\left(S_i\in (0,q]\right)\leq U(q).
		\end{align}
		Therefore,
		\begin{align}\label{step45}
			\lim_{x\to+\infty} J_{2,N}(x)=0.
		\end{align}
		Now combining \eqref{step44} and \eqref{step45},
		taking $x\to+\infty$ in \eqref{4grtgt6y6} yields that for any $N\geq1$,
		\begin{align} \label{4tgt6hy6}
			&\limsup_{x\to+\infty}\left|V(p(x)+q)-V(p(x))-\frac{\rho q}{\widehat{\mathbb{E}}_0[S_1]}\right|\cr
			&\leq U(q)\sum^{\infty}_{i=N+1}
			\widehat{\mathbb{P}}_0\left(S_i\leq0\right)
			+\left|\frac{q}{\widehat{\mathbb{E}}_0[S_1]}{\widehat{\mathbb{P}}_0\Big(\min_{1\leq i\leq N}S_i>0\Big)}-\frac{\rho q}{\widehat{\mathbb{E}}_0[S_1]}\right|.
		\end{align}
		Letting $N\to\infty$ in \eqref{4tgt6hy6}, by \eqref {45gtzy6uy6}
		and the definition of $\rho$,
		we get \eqref{4ty7us2}.

		\par
		Finally, plugging  \eqref{54trh62w} and \eqref{4ty7us2} into \eqref{4yhyuo9q1} gives that
		\begin{align}\label{5hhyhy5h5}
			&\lim_{x\to+\infty}\sum^{\infty}_{n=0}\widehat{\mathbb{P}}_{-x}( \overline{S}_n \leq 0, S_n\in (b,a] ) \cr
			&=\lim_{x\to+\infty}\int_{(a,0]}V(x+a-y)-V(x+b-y)\mathcal{U}(\mathrm{d}y)
			+\lim_{x\to+\infty}
			\int_{(b,a]}V(x)-V(x+b-y)\mathcal{U}(\mathrm{d}y)\cr
			&=\int_{(a,0]}\frac{\rho  (a-b)}{\widehat{\mathbb{E}}_0[S_1]}\mathcal{U}(\mathrm{d}y)+\int_{(b,a]}\frac{\rho  (y-b)}{\widehat{\mathbb{E}}_0[S_1]}\mathcal{U}(\mathrm{d}y)\cr
			&=\frac{(a-b)}{\widehat{\mathbb{E}}_0[S_1]}\widehat{\mathbb{P}}_0\Big(I_\infty \in (a,0]\Big)+\frac{1  }{\widehat{\mathbb{E}}_0[S_1]}\int_{(b,a]}(y-b)\widehat{\mathbb{P}}_0\Big(I_\infty \in \mathrm{d}y\Big),
		\end{align}
		where in the second equality, we used the dominated convergence theorem and Lemma \ref{4rtgt5gt} (i). Finally, according to Fubini's theorem, for any finite measure $\mu$ on $\mathbb{R}$, we have
		\begin{align}
			(a-b) \mu((a,0]) + \int_{(b,a]}(y-b) \mu(\mathrm{d}y) &= 	(a-b) \mu((a,0])+ \int_{(b,a]}\int_{(b, y]}1 \mathrm{d}z  \mu(\mathrm{d}y)\nonumber\\
			& = \int_{(b,a]}  \mu((a,0]) \mathrm{d}z+ \int_{(b,a]}   \mu([z, a]) \mathrm{d}z  =  \int_{(b,a]}   \mu([0, z]) \mathrm{d}z.
		\end{align}
		This, combined with \eqref{5hhyhy5h5}, concludes the
		proof of (i).
		
		(ii) The proof of (ii) is quite similar with $x$ replaced by $kh$ and $(b,a]$ replaced by $ih$ in the proof of (i). Also, we use
		Lemma \ref{Lem-renewal-theorem} (ii) in the lattice case.
	\end{proof}

	\par
	The following lemma gives two kinds of renewal theorems of the random walk $\{S_n\}_{n\geq0}$, which is cruical in proving Theorem \ref {456hy5t1}. Let $f$ be a non-negative Borel function.  For any $\kappa >0$,  define for $x\in\mathbb{R}$,
	\begin{align}
		&\overline{f}_\kappa (x):= \sum_{i \in \mathbb{Z}} 1_{(i\kappa, (i+1)\kappa]}(x) \sup_{z\in (i\kappa, (i+1)\kappa]} f(z),\nonumber\\
		&\underline{f}_\kappa (x):= \sum_{i \in \mathbb{Z}} 1_{(i\kappa, (i+1)\kappa]} (x)\inf_{z\in (i\kappa, (i+1)\kappa]} f(z).
	\end{align}
	We say that $f$ is directly Riemann integrable if $\int_{\mathbb R} \overline{f}_\kappa (x) \mathrm{d} x<\infty$ for
	some $\kappa>0$ and
	\begin{align}\label{DRI}
		\lim_{\kappa\to 0}\int_{\mathbb R}\overline{f}_\kappa (x)-\underline{f}_\kappa (x)\mathrm{d} x=0.
	\end{align}
	Denote by ${\rm DRI}_+(\mathbb{R})$ the family of  all non-negative directly Riemann integrable functions.

	\begin{lemma}\label{lem1}
		Assume
		$\mathbb{E}[Xe^{\gamma X}]<\infty$.
		
		\begin{itemize}
			\item[(i)]  If $X$ is non-lattice, then for any $\theta>0$,
			\begin{align}\label{eq1}
				\lim_{x\to+\infty} \sum_{\ell=1}^\infty  \widehat{\mathbb{E}}_{-x}\left(e^{-\theta S_\ell } \ind_{\{ \overline{S}_{\ell-1} \leq 0<S_\ell  \}}\right)
				=\lim_{x\to+\infty}  \widehat{\mathbb{E}}_{-x}\left(e^{-\theta S_{T_0^+} } \right)=
				\frac{1-\widehat{\mathbb{E}}_0[e^{-\theta \mathcal{H}_1 }]}{\theta \widehat{\mathbb{E}}_0[\mathcal{H}_1]}
			\end{align}
			and  for any  $f\in  {\rm DRI}_+(\mathbb{R})$ with compact support,
			\begin{align}\label{eq2}
				& \lim_{x\to+\infty} \sum_{\ell=1}^\infty  \widehat{\mathbb{E}}_{-x}\left(f(S_\ell)\ind_{\{\overline{S}_\ell \leq 0\}} \right)=\frac{1}{\widehat{\mathbb{E}}_0[S_1]}
				\widehat{\mathbb{E}}_0 \left[ \int_{-\infty}^{I_\infty}  f(z)\mathrm{d}z\right].
			\end{align}
			\item[(ii)]  If $X$ is lattice with span $h>0$, then for any $\theta>0$,
			\begin{align}\label{eq1'}
				\lim_{k\to+\infty} \sum_{\ell=1}^\infty  \widehat{\mathbb{E}}_{-kh}\left(e^{-\theta S_\ell } \ind_{\{ \overline{S}_{\ell-1} \leq 0<S_\ell  \}}\right)
				= 	\lim_{k\to+\infty}  \widehat{\mathbb{E}}_{-kh}\left(e^{-\theta S_{T_0^+} } \right)=
				\frac{h(1-\widehat{\mathbb{E}}_0[e^{-\theta \mathcal{H}_1 }])}{\widehat{\mathbb{E}}_0[\mathcal{H}_1](e^{\theta h}-1)}\quad
			\end{align}
			and for any non-negative function $f$ with compact support,
			\begin{align}\label{eq2'}
				& \lim_{k\to+\infty} \sum_{\ell=1}^\infty  \widehat{\mathbb{E}}_{-kh}\left(f(S_\ell)\ind_{\{\overline{S}_\ell\leq 0\}} \right)= \frac{h}{\widehat{\mathbb{E}}_0[S_1]}
				\widehat{\mathbb{E}}_0\left[\sum_{i=-\infty}^{I_\infty /h} f(ih)\right].
			\end{align}
		\end{itemize}
	\end{lemma}
	\begin{proof}
		We first prove \eqref{eq1} and \eqref{eq1'}.
		We only prove \eqref{eq1} here since the proof of \eqref{eq1'} is similar.
		Recall the definition of $T_x^+$ in \eqref{Def-of-T+}. Then
		\begin{align}\label{Eq1}
			\sum_{\ell=1}^\infty  \widehat{\mathbb{E}}_{-x}\left(e^{-\theta S_\ell } \ind_{\{ \overline{S}_{\ell-1} \leq 0<S_\ell  \}}\right) =  \sum_{\ell=1}^\infty  \widehat{\mathbb{E}}_{-x}\left(e^{-\theta S_{T_0^+} } \ind_{\{ T_0^+=\ell   \}}\right)= \widehat{\mathbb{E}}_{-x}\left(e^{-\theta S_{T_0^+} } \right)
			,
		\end{align}
		this implies the first equality.
		Set $\mathcal{T}_0:=0$. Denote by $\mathcal{T}_n:= \inf\left\{k>\mathcal{T}_{n-1}: S_k> S_{\mathcal{T}_{n-1}}\right\},~n\geq1$ the $n$-th ladder epoch of $\{S_k\}_{k\geq0}$.  Let $\mathcal{H}_n:= S_{\mathcal{T}_n},~n\geq0$ be the ladder hight process and $U^+(x):= \sum_{n=0}^\infty \widehat{\mathbb{P}}_0(\mathcal{H}_n\leq x)$ be the corresponding renewal function. Noticing that $\mathcal{H}_{n+1}-\mathcal{H}_{n},~n\geq0$ are i.i.d.	equal in law to
		$\mathcal{H}_1$	(see \cite[p.391-392]{feller1971}).
		By the strong Markov property of $\{S_k\}_{k\geq0}$, we have
		\begin{align}\label{step33}
			\widehat{\mathbb{E}}_{-x}\left(e^{-\theta S_{T_0^+} } \right)
			& = \sum_{n=0}^\infty  \widehat{\mathbb{E}}_{0}\left(e^{-\theta (S_{T_x^+}-x) } \ind_{\{ \mathcal{H}_n\leq x < \mathcal{H}_{n+1} \}}\right) =\sum_{n=0}^\infty  \widehat{\mathbb{E}}_{0}\left(e^{-\theta (\mathcal{H}_{n+1}-x) } \ind_{\{ \mathcal{H}_n\leq x < \mathcal{H}_{n+1} \}}\right)\cr
			&=
			\sum_{n=0}^\infty \widehat{\mathbb{E}}_{0} \left(\ind_{\{ \mathcal{H}_n\leq x\}} e^{-\theta(\mathcal{H}_n-x)}\widehat{\mathbb{E}}_{0}\big(e^{-\theta \mathcal{H}_{1} } \ind_{\{ x -y< \mathcal{H}_{1}\}}\big)\big|_{y=\mathcal{H}_n}\right)\cr
			& = \int_{[0,x]} e^{\theta (x-y)} \widehat{\mathbb{E}}_0\left(e^{-\theta \mathcal{H}_1}\ind_{\{\mathcal{H}_1> x-y\}}\right)U^+(\mathrm{d}y),
		\end{align}
		where the second equality follows from the fact on the event ${\{ \mathcal{H}_n\leq x < \mathcal{H}_{n+1} \}}$, we have $\mathcal{T}_n=T_x^+$.
		For each $N>2$ and $\delta>0$, firstly noticing that
		\begin{align}
			&\sup_{x>N} \int_{[0, x-N]} e^{\theta (x-y)} \widehat{\mathbb{E}}_0\left(e^{-\theta \mathcal{H}_1}\ind_{\{\mathcal{H}_1> x-y\}}\right)U^+(\mathrm{d}y)\nonumber\\
			&\leq \sup_{x>N} \int_{[0, x-N]}  \widehat{\mathbb{P}}_0\left(\mathcal{H}_1> x-y\right)U^+(\mathrm{d}y) = \sup_{x>N} \int_N ^\infty U^+(x-N)-U^+(x-z)  \widehat{\mathbb{P}}_0\left(\mathcal{H}_1\in \mathrm{d}z\right) \nonumber\\
			&\leq \int_{[N,\infty]} U^+(z-N) \widehat{\mathbb{P}}_0\left(\mathcal{H}_1\in \mathrm{d}z\right)  \leq  \widehat{\mathbb{E}}_0 (U^+ (\mathcal{H}_1) \ind_{\{\mathcal{H}_1\geq N\}}).
		\end{align}
		Since $\widehat{\mathbb{E}}_0[\mathcal{H}_1]<\infty$ (see \cite[p.397, Theorem 2 (ii)]{feller1971}), by Lemma \ref{4rtgt5gt} (ii) and the dominated convergence theorem, above yields that
		\begin{align}\label{step11}
			\lim_{N\to\infty}\sup_{x>N} \int_{[0, x-N]} e^{\theta (x-y)} \widehat{\mathbb{E}}_0\left(e^{-\theta \mathcal{H}_1}\ind_{\{\mathcal{H}_1> x-y\}}\right)U^+(\mathrm{d}y)=0.
		\end{align}
		\par
		Next, we address $y\in (x-N, x]$. By Lemma \ref{Lem-renewal-theorem}, it follows that
		\begin{align}
			&\limsup_{x\to\infty}   \int_{(x-N,x]} e^{\theta (x-y)} \widehat{\mathbb{E}}_0\left(e^{-\theta \mathcal{H}_1}\ind_{\{\mathcal{H}_1> x-y\}}\right)U^+(\mathrm{d}y)\nonumber\\
			& \leq \limsup_{x\to\infty}  \sum_{j=1}^{\lfloor N/\delta\rfloor +1} e^{\theta j\delta} \left(U^+(x-(j-1)\delta)-U^+(x-j\delta)\right) \widehat{\mathbb{E}}_0\left(e^{-\theta \mathcal{H}_1}\ind_{\{\mathcal{H}_1> (j-1)\delta\}}\right)\nonumber\\
			& = \frac{1}{\widehat{\mathbb{E}}_0[\mathcal{H}_1]} \sum_{j=1}^{\lfloor N/\delta \rfloor+1} e^{\theta j\delta} \delta  \widehat{\mathbb{E}}_0\left(e^{-\theta \mathcal{H}_1}\ind_{\{\mathcal{H}_1>  (j-1)\delta\}}\right)\cr
			& \stackrel{\delta\downarrow 0+}{\longrightarrow}  \frac{1}{\widehat{\mathbb{E}}_0[\mathcal{H}_1]}  \int_0^N e^{\theta z}   \widehat{\mathbb{E}}_0\left(e^{-\theta \mathcal{H}_1}\ind_{\{\mathcal{H}_1> z\}}\right) \mathrm{d}z.
		\end{align}
		Also, for the lower bound, similarly,
		\begin{align}
			&\liminf_{x\to\infty}   \int_{(x-N,x]} e^{\theta (x-y)} \widehat{\mathbb{E}}_0\left(e^{-\theta \mathcal{H}_1}\ind_{\{H_1> x-y\}}\right)U^+(\mathrm{d}y)\nonumber\\
			& \geq \liminf_{x\to\infty}  \sum_{j=1}^{\lfloor N/\delta\rfloor } e^{\theta (j-1)\delta} \left(U^+(x-(j-1)\delta)-U^+(x-j\delta)\right) \widehat{\mathbb{E}}_0\left(e^{-\theta \mathcal{H}_1}\ind_{\{\mathcal{H}_1\geq  j \delta\}}\right)\nonumber\\
			& = \frac{1}{\widehat{\mathbb{E}}_0[\mathcal{H}_1]} \sum_{j=1}^{\lfloor N/\delta \rfloor} e^{\theta (j-1)\delta} \delta  \widehat{\mathbb{E}}_0\left(e^{-\theta \mathcal{H}_1}\ind_{\{\mathcal{H}_1\geq  j\delta\}}\right)\cr
			&\stackrel{\delta\downarrow 0+}{\longrightarrow}  \frac{1}{\widehat{\mathbb{E}}_0[\mathcal{H}_1]}  \int_0^N e^{\theta z}   \widehat{\mathbb{E}}_0\left(e^{-\theta \mathcal{H}_1}\ind_{\{\mathcal{H}_1\geq  z\}}\right) \mathrm{d}z.
		\end{align}
		Since $\int_0^N e^{\theta z}   \widehat{\mathbb{E}}_0\left(e^{-\theta \mathcal{H}_1}\ind_{\{H_1=  z\}}\right) \mathrm{d}z=0$, we conclude that
		\begin{align}\label{step12}
			&\lim_{x\to\infty}   \int_{(x-N,x]} e^{\theta (x-y)} \widehat{\mathbb{E}}_0\left(e^{-\theta \mathcal{H}_1}\ind_{\{\mathcal{H}_1> x-y\}}\right)U^+(\mathrm{d}y)\nonumber\\
			& =  \frac{1}{\widehat{\mathbb{E}}_0[\mathcal{H}_1]}  \int_0^N e^{\theta z}   \widehat{\mathbb{E}}_0\left(e^{-\theta \mathcal{H}_1}\ind_{\{\mathcal{H}_1> z\}}\right) \mathrm{d}z.
		\end{align}
		Therefore, by Fubini's theorem, we have
		\begin{align}
			&\lim_{x\to+\infty} \sum_{\ell=1}^\infty  \widehat{\mathbb{E}}_{-x}\left(e^{-\theta S_\ell } \ind_{\{ \max_{j\leq \ell-1} S_j \leq 0<S_\ell  \}}\right)\stackrel{\eqref{step33}}{=}\lim_{x\to+\infty}  \int_{[0,x]} e^{\theta (x-y)} \widehat{\mathbb{E}}_0\left(e^{-\theta \mathcal{H}_1}\ind_{\{\mathcal{H}_1> x-y\}}\right)U^+(\mathrm{d}y)\nonumber\\
			&\stackrel{\eqref{step11}}{=} \lim_{N\to\infty} \lim_{x\to\infty} \int_{(x-N,x]} e^{\theta (x-y)} \widehat{\mathbb{E}}_0\left(e^{-\theta \mathcal{H}_1}\ind_{\{\mathcal{H}_1> x-y\}}\right)U^+(\mathrm{d}y) \nonumber\\
			& \stackrel{\eqref{step12}}{=}\frac{1}{\widehat{\mathbb{E}}_0[\mathcal{H}_1]}  \int_0^\infty  e^{\theta z}   \widehat{\mathbb{E}}_0\left(e^{-\theta \mathcal{H}_1}\ind_{\{\mathcal{H}_1> z\}}\right) \mathrm{d}z= \frac{1}{\widehat{\mathbb{E}}_0[\mathcal{H}_1]} \widehat{\mathbb{E}}_0\left[e^{-\theta \mathcal{H}_1}\int^{\mathcal{H}_1}_{0}e^{\theta z}\mathrm{d}z\right]\cr
			&=\frac{1-\widehat{\mathbb{E}}_0[e^{-\theta \mathcal{H}_1 }]}{\theta \widehat{\mathbb{E}}_0[\mathcal{H}_1]},
		\end{align}
		which completes the proof of \eqref{eq1}.
		\par

		Now we
		prove \eqref{eq2} and \eqref{eq2'}. Since \eqref{eq2'} follows immediately from Lemma \ref{lem2}, it suffices to prove \eqref{eq2}. Suppose that the support of $f$ is a subset of $[-N, N]$.
		For each $\delta>0$, set $N_\delta:= \lfloor N/\delta \rfloor +1$, then
		we have the following lower bound
		\begin{align}
			\sum_{\ell=1}^\infty  \widehat{\mathbb{E}}_{-x}\left( f(S_\ell)\ind_{\{\overline{S}_\ell \leq 0\}} \right)  &\geq \sum^{\infty}_{\ell =1}\sum^{0}_{j=-N_\delta+1}\inf_{y\in((j-1)\delta,j\delta]}f(y)\widehat{\mathbb{P}}_{-x}(\overline{S}_\ell \leq 0, S_\ell \in((j-1)\delta,j\delta])\cr
			&= \sum^{0}_{j=-N_\delta+1}\inf_{y\in((j-1)\delta,j\delta]}f (y)\sum^{\infty}_{\ell=1}\widehat{\mathbb{P}}_{-x}(\overline{S}_\ell \leq 0,  S_\ell\in((j-1)\delta,j\delta]).
		\end{align}
		Taking $x\to+\infty$ in the above inequality, by Lemma \ref{lem2},
		\begin{align}\label{step26}
			&\liminf_{x\to+\infty}  \sum_{\ell=1}^\infty  \widehat{\mathbb{E}}_{-x}\left(f(S_\ell)\ind_{\{\overline{S}_\ell \leq 0\}} \right) \nonumber\\
			&\geq  \frac{1}{\widehat{\mathbb{E}}_0[S_1]}\sum^{0}_{j=-N_\delta+1}\inf_{y\in((j-1)\delta,j\delta]}f(y) \int_{((j-1)\delta, j\delta]}\widehat{\mathbb{P}}_0(I_\infty\in [z,0])\mathrm{d}z.
		\end{align}
		Noticing that for each $\delta>0$,
		\begin{align}\label{step27}
			&	\left| \sum^{0}_{j=-N_\delta+1}\inf_{y\in((j-1)\delta,j\delta]}f (y) \int_{((j-1)\delta, j\delta]}\widehat{\mathbb{P}}_0(I_\infty\in [z,0])\mathrm{d}z - \int_{-\infty}^0 f(z)\widehat{\mathbb{P}}_0(I_\infty\in [z,0])\mathrm{d} z \right| \nonumber\\
			& =  \sum^{0}_{j=-N_\delta+1} \int_{((j-1)\delta, j\delta]} \left(f(z) - \inf_{y\in((j-1)\delta,j\delta]}f (y) \right)\widehat{\mathbb{P}}_0(I_\infty\in [z,0])\mathrm{d}z  \nonumber\\
			&\leq \sum^{0}_{j=-N_\delta+1} \int_{((j-1)\delta, j\delta]} \left(\sup_{y\in((j-1)\delta,j\delta]} f(y) - \inf_{y\in((j-1)\delta,j\delta]}f(y) \right)\mathrm{d}z \nonumber\\
			& = \int_{\mathbb{R}} \left( \overline{f}_\delta(z)- \underline{f}_\delta(z) \right)\mathrm{d} z.
		\end{align}
		Since $f\in {\rm DRI}_+(\mathbb{R})$, combining \eqref{step26} and \eqref{step27}, taking $\delta \to 0$, we conclude that
		\begin{align}\label{step30}
			&\liminf_{x\to+\infty}  \sum_{\ell=1}^\infty  \widehat{\mathbb{E}}_{-x}\left(f(S_\ell)\ind_{\{\overline{S}_\ell \leq 0\}} \right) \geq \frac{1}{\widehat{\mathbb{E}}_0[S_1]}\int_{-\infty}^0 f(z)\widehat{\mathbb{P}}_0(I_\infty\in [z,0])\mathrm{d} z.
		\end{align}
		For the upper bound, repeating the same argument yields that
		\begin{align}\label{step31}
			&\limsup_{x\to+\infty}  \sum_{\ell=1}^\infty  \widehat{\mathbb{E}}_{-x}\left(f(S_\ell)\ind_{\{\overline{S}_\ell \leq 0\}} \right) \leq \frac{1}{\widehat{\mathbb{E}}_0[S_1]}\int_{-\infty}^0 f(z)\widehat{\mathbb{P}}_0(I_\infty\in [z,0])\mathrm{d} z.
		\end{align}
		Combining \eqref{step30}, \eqref{step31} and Fubini's theorem, we get \eqref{eq2}.
	\end{proof}
	
	According to Lemma \ref{lem1}, we have the following corollary.
	\begin{corollary}\label{Cor:tightness-overshoot}
		Assume $\mathbb{E}[Xe^{\gamma X}]<\infty$. Then $\{(S_{T_0^+}, \widehat{\mathbb{P}}_{-x}), x> 0\}$ is tight. That is,
		\begin{align}
			\lim_{K\to+\infty} \sup_{x>0} \widehat{\mathbb{P}}_{-x}\left(S_{T_0^+}>K\right)=0.
		\end{align}
	\end{corollary}
	\begin{proof}
		We only prove the case that $X$ is non-lattice here since the proof for the lattice case is similar.  Noticing that for any $K>0$, the following holds:
		\begin{align}\label{step40}
			\widehat{\mathbb{E}}_{-x}\left(1-e^{- S_{T_0^+}/K}\right)  & \geq (1-e^{-1}) \widehat{\mathbb{P}}_{-x}\left(S_{T_0^+}>K\right) .
		\end{align}
		Also,
		from Lemma \ref{lem1} (i) and \eqref{Eq1}, it holds that
		\begin{align}\label{step41}
			\lim_{K\to +\infty}	\lim_{x\to+\infty} \widehat{\mathbb{E}}_{-x}\left(1-e^{- S_{T_0^+}/K}\right) &=  1- \lim_{K\to +\infty}	 	\frac{K\left(1-\widehat{\mathbb{E}}_0[e^{-\mathcal{H}_1 /K}]\right)}{ \widehat{\mathbb{E}}_0[\mathcal{H}_1]} \nonumber\\
			& = 0.
		\end{align}
		Therefore, combining \eqref{step40} and \eqref{step41}, for any $\varepsilon>0$, there 
		exist
		$K_\varepsilon>0$ and $N_\varepsilon$ such that
		\begin{align}\label{step42}
			\sup_{x>N_\varepsilon} \widehat{\mathbb{P}}_{-x}\left(S_{T_0^+}>K_\varepsilon \right)<\varepsilon.
		\end{align}
		For $x\in (0, N_\varepsilon]$, 	according to the definition of $T_0^+$, it holds that
		\begin{align}
			\widehat{\mathbb{E}}_{-x}\left(S_{T_0^+}\right)& =\sum_{n=1}^\infty \widehat{\mathbb{E}}_{-x}\left(S_n \ind_{\{ \overline{S}_{n-1}\leq 0< S_n\}}\right) \leq \sum_{n=1}^\infty \widehat{\mathbb{E}}_{-x}\left(|S_n-S_{n-1}| \ind_{\{ S_{n-1}\leq 0\}}\right) .
		\end{align}
		Since $S_n-S_{n-1}$ is independent of $\sigma(S_1,...,S_{n-1})$, we conclude from the above inequality that for any $x\in (0,N_\varepsilon]$,
		\begin{align}\label{step43}
			\widehat{\mathbb{E}}_{-x}\left(S_{T_0^+}\right) &  \leq \widehat{\mathbb{E}}_0		(|S_1|)
			\sum_{n=1}^\infty \widehat{\mathbb{P}}_{-x}(S_{n-1}\leq 0)= \widehat{\mathbb{E}}_0		(|S_1|)
			\sum_{n=0}^\infty \widehat{\mathbb{P}}_{0}(S_{n}\leq x)\nonumber\\
			&\leq  \widehat{\mathbb{E}}_0		(|S_1|)
			\sum_{n=0}^\infty \widehat{\mathbb{P}}_{0}(S_{n}\leq N_\varepsilon).
		\end{align}
		Therefore, combining \eqref{step42} and \eqref{step43}, we conclude that for any $K>K_\varepsilon$,
		\begin{align}
			\sup_{x>0} \widehat{\mathbb{P}}_{-x}\left(S_{T_0^+}>K\right) & \leq 	\sup_{x>N_\varepsilon} \widehat{\mathbb{P}}_{-x}\left(S_{T_0^+}>K_\varepsilon \right) + 	\sup_{x \in (0, N_\varepsilon]} \widehat{\mathbb{P}}_{-x}\left(S_{T_0^+}>K \right)\nonumber\\
			& <\varepsilon + \frac{1}{K}\sup_{x \in (0, N_\varepsilon]} \widehat{\mathbb{E}}_{-x}\left(S_{T_0^+}\right)\nonumber\\
			& \leq \varepsilon + \frac{1}{K}\widehat{\mathbb{E}}_0	(|S_1|)
			\sum_{n=0}^\infty \widehat{\mathbb{P}}_{0}(S_{n}\leq N_\varepsilon).
		\end{align}
		Since $\varepsilon>0$ is arbitrary, taking $K\to+\infty$ first and then $\varepsilon\to 0$ in the above inequality, we arrive at the desired result.
	\end{proof}

	The following renewal theorem will be used in studying the maximal displacement of branching random walk with killing.
	\begin{lemma}\label{lem3}
		Assume  $\mathbb{E}[X e^{\gamma X}]<\infty$.
		\begin{itemize}
			\item[(i)] 	If $X$ is non-lattice,
			then for any $y>0$,
			\begin{align}\label{eq3}
				\lim_{x\to+\infty} \sum_{\ell=1}^\infty    \widehat{\mathbb{E}}_{y-x}\Big( e^{-\gamma S_\ell}\ind_{\{S_j \in (-x, 0], 1\leq j\leq \ell-1 , S_\ell >0\}} \Big) = \widehat{\mathbb{P}}_{0}\left(I_\infty > -y\right)  \frac{1-\widehat{\mathbb{E}}_0[e^{-\gamma \mathcal{H}_1 }]}{\gamma \widehat{\mathbb{E}}_0[\mathcal{H}_1]}.
			\end{align}
			\item[(ii)] 	If $X$ is lattice with span $h>0$,
			then for any $i\geq1$,
			\begin{align}\label{eq3'}
				\lim_{k\to+\infty} \sum_{\ell=1}^\infty    \widehat{\mathbb{E}}_{(i-k)h}\Big( e^{-\gamma S_\ell}\ind_{\{S_j \in (-kh, 0], 1\leq j\leq \ell-1 , S_\ell >0\}} \Big) =\widehat{\mathbb{P}}_{0}\left(I_\infty >-ih\right) 	\frac{h(1-\widehat{\mathbb{E}}_0[e^{-\gamma \mathcal{H}_1 }])}{\widehat{\mathbb{E}}_0[\mathcal{H}_1](e^{\gamma h}-1)}.
			\end{align}
		\end{itemize}
	\end{lemma}
	\begin{proof}
		We only prove the non-lattice case here.
		For simplicity, for $N\geq 1$, define
		\begin{align}
			R_{1, N}(x)& :=  \sum_{\ell=1}^N     \widehat{\mathbb{E}}_{y-x}\Big( e^{-\gamma S_\ell}\ind_{\{S_j \in (-x, 0], 1\leq j\leq \ell-1 , S_\ell >0\}} \Big) ,
			\nonumber\\
			R_{2, N}(x) &:= \sum_{\ell=N+1}^\infty      \widehat{\mathbb{E}}_{y-x}\Big( e^{-\gamma S_\ell}\ind_{\{S_j \in (-x, 0], 1\leq j\leq \ell-1 , S_\ell >0\}} \Big)		.
		\end{align}
		Then,
		\begin{align}\label{Def-of-R}
			\sum_{\ell=1}^\infty    \widehat{\mathbb{E}}_{y-x}\Big( e^{-\gamma S_\ell}\ind_{\{S_j \in (-x, 0], 1\leq j\leq \ell-1 , S_\ell >0\}} \Big) = 	R_{1, N}(x) +	R_{2, N}(x).
		\end{align}
		For $R_{1,N}(x)$, noticing that for each fixed $N$,
		\begin{align}\label{step16}
			& \lim_{x\to+\infty} R_{1,N}(x)=  \lim_{x\to+\infty}  \sum_{\ell=1}^N     \widehat{\mathbb{E}}_{y-x}\Big( e^{-\gamma S_\ell}\ind_{\{S_j \in (-x, 0], 1\leq j\leq \ell-1 , S_\ell >0\}} \Big)  \nonumber\\
			& \leq  \lim_{x\to+\infty}   \sum_{\ell=1}^N   \widehat{\mathbb{P}}_{y-x}( S_\ell >0 ) =     \sum_{\ell=1}^N  \lim_{x\to+\infty} \widehat{\mathbb{P}}_{0}( S_\ell >x-y )=0.
		\end{align}
		Now we  treat $R_{2,N}(x)$. Define
		\begin{align}
			R_{3,N}(x) &:= \sum_{\ell=N+1}^\infty   \widehat{\mathbb{E}}_{y-x}\Big( e^{-\gamma S_\ell}\ind_{\{ \underline{S}_N  > -x, \overline{S}_{\ell-1} \leq 0< S_\ell \}} \Big) ,\nonumber\\
			R_{4,N}(x) &:=  \sum_{\ell=N+1}^\infty   \widehat{\mathbb{E}}_{y-x}\Big( e^{-\gamma S_\ell}\ind_{\{ \underline{S}_N > -x, \max_{N+1\leq  j\leq \ell-1} S_\ell \leq 0< S_\ell \}} \Big)  .
		\end{align}
		Set $T_{-x}^-:= \inf\left\{
		k\in \mathbb{N}:
		S_k \leq -x \right\}$. We mention here that
		$\widehat{\mathbb{P}}_y\left(T_0^-=+\infty \right)>0$.
		By the Markov property, we see that
		\begin{align}
			&\left|R_{2,N}(x)- R_{3,N}(x)\right|\nonumber\\
			& \leq \sum_{\ell=N+1}^\infty    \widehat{\mathbb{E}}_{y-x}\left( e^{-\gamma S_\ell}  (\ind_{\{ \underline{S}_N > -x, \overline{S}_{\ell-1}\leq 0< S_\ell \}} - \ind_{\{S_j \in (-x, 0], 1\leq j\leq \ell-1 , S_\ell >0\}}   ) \right) \nonumber\\
			&\leq   \sum_{\ell=N+1}^\infty    \widehat{\mathbb{E}}_{y-x}\left(e^{-\gamma S_\ell}\ind_{\{T_{-x}^-\in [N+1, \ell-1] , S_\ell>0\}} \right)\nonumber\\
			&\leq  \sum_{k=1}^\infty e^{-\gamma (k-1)}\sum_{\ell=N+1}^\infty    \widehat{\mathbb{P}}_{y}\left( T_0^-\in [N+1, \ell-1], S_\ell\in (x+k-1, x+k] \right).
		\end{align}
		Combining the Markov property and Lemma \ref{4rtgt5gt} (i), the above difference has upper bound
		\begin{align}\label{step14}
			&\left|R_{2,N}(x)- R_{3,N}(x)\right|\nonumber\\
			& \leq  \sum_{k=1}^\infty e^{-\gamma (k-1)}\sum_{\ell=N+1}^\infty    \sum_{q=N+1}^{\ell-1} \widehat{\mathbb{P}}_{y}\left( T_0^- =q,  S_\ell\in (x+k-1, x+k] \right) \nonumber\\
			& = \sum_{k=1}^\infty e^{-\gamma (k-1)} \sum_{q=N+1}^{\infty}    \widehat{\mathbb{E}}_{y}\left( \ind_{\{T_0^- =q\} }  \sum_{\ell=1}^\infty
			\widehat{\mathbb{P}}_{z} \left( S_\ell\in (x+k-1, x+k] \right)\Big|_{z=S_q}\right) \nonumber\\
			& \leq U(1)  \sum_{k=1}^\infty e^{-\gamma (k-1)} \sum_{q=N+1}^{\infty}    \widehat{\mathbb{P}}_y(T_0^- =q) = \frac{U(1)}{1-e^{-\gamma}}
			\widehat{\mathbb{P}}_y(T_0^- \in (N, +\infty)).
		\end{align}
		Also, noticing that
		\begin{align}
			&\left|R_{3,N}(x)- R_{4,N}(x)\right|\nonumber\\
			&\leq  \sum_{\ell=N+1}^\infty  \left|     \widehat{\mathbb{E}}_{y-x}\Big( e^{-\gamma S_\ell}\ind_{\{ \underline{S}_N> -x, \overline{S}_{\ell-1} \leq 0< S_\ell \}} \Big) -   \widehat{\mathbb{E}}_{y-x}\Big( e^{-\gamma S_\ell}\ind_{\{ \underline{S}_N > -x, \max_{N+1\leq  j\leq \ell-1} S_\ell \leq 0< S_\ell \}} \Big)  \right| \nonumber\\
			& =  \sum_{\ell=N+1}^\infty      \widehat{\mathbb{E}}_{y-x}\left( e^{-\gamma S_\ell}  \Big( \ind_{\{ \underline{S}_N > -x, \max_{N+1\leq  j\leq \ell-1} S_\ell \leq 0< S_\ell \}}-\ind_{\{ \underline{S}_N > -x, \overline{S}_{\ell-1} \leq 0< S_\ell \}}   \Big) \right) \nonumber\\
			& \leq  \sum_{\ell=N+1}^\infty      \widehat{\mathbb{E}}_{y-x}\left( e^{-\gamma S_\ell}   \ind_{\{  \overline{S}_N  >0, S_\ell>0\}}  \right)  \leq \sum_{k=1}^\infty e^{-\gamma(k-1)} \sum_{\ell=N+1}^\infty \widehat{\mathbb{P}}_{y-x} \left(  \overline{S}_N >0, S_\ell \in (k-1, k] \right).
		\end{align}
		Similarly combining the Markov property and Lemma  \ref{4rtgt5gt} (i), we obtain that
		\begin{align}\label{step15}
			&\left|R_{3,N}(x)-    R_{4,N}(x)\right|\nonumber\\
			&  \leq \sum_{k=1}^\infty e^{-\gamma(k-1)} \widehat{\mathbb{E}}_{y-x} \left( \ind_{\{ \overline{S}_N >0 \}}  \sum_{\ell=1}^\infty
			\widehat{\mathbb{P}}_{z}( S_\ell \in (k-1, k])\big|_{z=S_N} \right)\nonumber\\
			&\leq \frac{U(1) }{1-e^{-\gamma}} \widehat{\mathbb{P}}_{y-x}( \overline{S}_N  >0).
		\end{align}
		Therefore, combining \eqref{step14} and \eqref{step15}, we conclude that
		\begin{align}\label{step17}
			& \limsup_{x\to+\infty}  \left|R_{2,N}(x)- R_{4,N}(x)\right|
			\leq \frac{U(1)}{1-e^{-\gamma}} \widehat{\mathbb{P}}_y(T_0^- \in (N, +\infty)).
		\end{align}
		Finally, combining the Markov property and Lemma \ref{lem1}, we see that
		\begin{align}\label{step18}
			& \lim_{x\to+\infty} R_{4,N}(x)=  \lim_{x\to+\infty}  \sum_{\ell=N+1}^\infty   \widehat{\mathbb{E}}_{y-x}\big( e^{-\gamma S_\ell}\ind_{\{ \underline{S}_N > -x, \max_{N+1\leq  j\leq \ell-1} S_\ell \leq 0< S_\ell \}}\big )  \nonumber\\
			& = \lim_{x\to+\infty}   \widehat{\mathbb{E}}_{y-x}\left( \ind_{\{ \underline{S}_N  > -x\}}  \sum_{\ell=1}^\infty \widehat{\mathbb{E}}_{z}  \Big(e^{-\gamma S_\ell}\ind_{\{ \overline{S}_{\ell-1} \leq 0< S_\ell \}} \Big)\Big|_{z=S_N}\right)  \nonumber\\
			& = \lim_{x\to+\infty}   \widehat{\mathbb{E}}_{y}\left( \ind_{\{ \underline{S}_N > 0 \}} \sum_{\ell=1}^\infty  \widehat{\mathbb{E}}_{z}  \Big(e^{-\gamma S_\ell}\ind_{\{ \overline{S}_{\ell-1} \leq 0< S_\ell \}} \Big)\Big|_{z=S_N-x}\right) \nonumber\\
			& =  \frac{1-\widehat{\mathbb{E}}_0[e^{-\gamma \mathcal{H}_1 }]}{\gamma \widehat{\mathbb{E}}_0[\mathcal{H}_1]} \widehat{\mathbb{P}}_{y}\Big(\min_{j\leq N} S_j > 0\Big) \stackrel{N\to+\infty}{\longrightarrow}  \widehat{\mathbb{P}}_{y}\Big(\min_{j\geq 0} S_j > 0\Big) \frac{1-\widehat{\mathbb{E}}_0[e^{-\gamma \mathcal{H}_1 }]}{\gamma \widehat{\mathbb{E}}_0[\mathcal{H}_1]} .
		\end{align}
		Therefore, we conclude that
		\begin{align}
			& \lim_{x\to+\infty} \sum_{\ell=1}^\infty    \widehat{\mathbb{E}}_{y-x}\big( e^{-\gamma S_\ell}\ind_{\{S_j \in (-x, 0], 1\leq j\leq \ell-1 , S_\ell >0\}} \big) = \lim_{x\to +\infty} (R_{1,N}(x)+ R_{2,N}(x))\nonumber\\
			& \stackrel{\eqref{step16}} {=}\lim_{N\to+\infty} \lim_{x\to +\infty}  R_{2,N}(x) \stackrel{\eqref{step17}}{=} \lim_{N\to+\infty} \lim_{x\to +\infty}  R_{4,N}(x) \nonumber\\
			& \stackrel{\eqref{step18}}{=}  \widehat{\mathbb{P}}_{0}\left(I_\infty > -y\right) \frac{1-\widehat{\mathbb{E}}_0[e^{-\gamma \mathcal{H}_1 }]}{\gamma \widehat{\mathbb{E}}_0[\mathcal{H}_1]} ,
		\end{align}
		which implies \eqref{eq3}.
	\end{proof}

	\section{Proof of Theorem \ref{456hy5t1}: Tail probability of $M$}

	Recall that ${\rm DRI}_+(\mathbb{R})$ stands for the family of all non-negative directly Riemann integrable functions; see \eqref{DRI}. The following lemma shows that ${\rm DRI}_+(\mathbb{R})$ is closed under the multiplication of monotonic functions.
	\begin{lemma}\label{lem4}
		Suppose that $G$ and $H$ are two non-negative functions on $\mathbb{R}$ and that both of $G$ and $H$ are monotone. Then for any $a<b$,
		$G(x)H(x)\ind_{[a,b]}(x)\in {\rm DRI}_+(\mathbb{R})$.
	\end{lemma}
	\begin{proof}
		If both of $G$ and $H$ are increasing (resp. decreasing), then $GH$ is increasing (resp. decreasing). Noticing that $1$ is decreasing and that $GH\cdot 1= G\cdot H$, thus without loss of generality, we assume that $G$ is increasing and that $H$ is decreasing. Also, set $F(x):=G(x)H(x)\ind_{[a,b]}(x)$.
		
		Since  $G$ and $H$ are monotonic, let $K$ be the constant such that $G(b+1)+H(a-1)\leq K<\infty$.  Then, for any $\kappa\in (0,1)$,
		\begin{align}
			\overline{F}_\kappa(x)= \sum_{i\in \mathbb{Z}} \ind_{(i\kappa, (i+1)\kappa]}(x) \sup_{z\in (i\kappa, (i+1)\kappa]} F(z)\leq K^2\ind_{[a-1,b+1]}(x),
		\end{align}
		which implies that $\int_{\mathbb{R}}\overline{F}_\kappa(x) \mathrm{d}x<\infty$.
		
		For each $i\in\mathbb{Z}$ such that $ (i\kappa, (i+1)\kappa] \subset  [a,b]$, it holds that
		for any $z,y\in (i\kappa, (i+1)\kappa]$,
		\begin{align}
			\left|F(z)-F(y)\right|&=	\left|G(z)H(z)-G(y)H(y)\right|\cr
			&\leq G(z)\left|H(z)-H(y)\right| + H(y)\left|G(z)-G(y)\right|\nonumber\\
			&\leq K\Big(\left|H(z)-H(y)\right| + \left|G(z)-G(y)\right|\Big)\nonumber\\
			&\leq K\Big(H(i\kappa)-H((i+1)\kappa) + G((i+1)\kappa)-G(i\kappa) \Big),
		\end{align}
		where the last equality follows from the monotonicity of $G$ and $H$. Thus, for each $i\in\mathbb{Z}$ such that $ (i\kappa, (i+1)\kappa] \subset  [a,b]$, we have
		\begin{align}\label{fr4tt5t5}
			\sup_{z\in (i\kappa, (i+1)\kappa]} F(z)- \inf_{z\in (i\kappa, (i+1)\kappa]} F(z)\leq K\Big(H(i\kappa)-H((i+1)\kappa) + G((i+1)\kappa)-G(i\kappa) \Big).
		\end{align}
		For each $(i\kappa, (i+1)\kappa]\cap [a,b]\neq \emptyset$ and $(i\kappa, (i+1)\kappa] \nsubseteq [a,b]$, we also have that
		\begin{align}\label{5huhy6de3}
			\sup_{z\in (i\kappa, (i+1)\kappa]} F(z)- \inf_{z\in (i\kappa, (i+1)\kappa]} F(z) {\leq}\sup_{z\in (i\kappa, (i+1)\kappa]} F(z)\leq K^2.
		\end{align}
		Let $j_1,j_2$ be constants such that $\{i\in \mathbb{Z}: (i\kappa, (i+1)\kappa] \subset [a,b]\} = \{i: j_1\leq i\leq j_2\}$. By \eqref{fr4tt5t5} and \eqref{5huhy6de3}, for each $\kappa>0$,
		\begin{align}\label{3fgteu7q}
			\int_{-\infty}^\infty \overline{F}_\kappa(x)-\underline{F}_\kappa(x) \mathrm{d}x & \leq 2 \kappa K^2+ K\kappa \sum_{i=j_1}^{j_2} \Big(H(i\kappa)-H((i+1)\kappa) + G((i+1)\kappa)-G(i\kappa) \Big) \nonumber\\
			& = 2 \kappa K^2+ K\kappa \Big(H(j_1\kappa)- H(j_2\kappa)+ G(j_2\kappa)-G(j_1\kappa)\Big) \leq 4\kappa K^2,
		\end{align}
		which implies that $\lim_{\kappa\to 0} 	 \int_{-\infty}^\infty \overline{F}_\kappa(x)-\underline{F}_\kappa(x) \mathrm{d}x =0$. Therefore, we complete the proof of the lemma.
	\end{proof}

	The following lemma gives several integrability results for $\psi$ under some mild conditions for the offspring law. Recall that
	\begin{align}\label{4gr6hytjy5}
		\psi(s)&=\Big(\sum^{\infty}_{n=0}p_n(1-s)^n\Big)+ms-1=\sum^{\infty}_{n=1}p_n((1-s)^n+ns-1).
	\end{align}
	By the Bernoulli inequality $(1+x)^n>1+nx$ for $x>-1$ and $n>0$, we have
	$$\psi(s)\geq 0,\qquad s\in[0,1].$$
	Moreover, $\psi$ is a continuous function in $[0,1]$ and that
	\begin{align}\label{psidaoshu}
		\psi'(s)= m-\sum_{n=1}^\infty np_n(1-s)^{n-1}\geq 0.
	\end{align}

	\begin{lemma}\label{43tt67i8de3}
		Suppose that $f$ is a non-negative
		measurable function
		on $\mathbb{R}$ such that $f (x)\leq \min\{1, e^{-\gamma x}\}$ for all $x\in \mathbb{R}$.
		If $\sum^{\infty}_{k=1}(k\log k)p_k<\infty$, then
		\[
		\int_0^\infty \psi(f(x)) e^{\gamma x}\mathrm{d}x<\infty.
		\]
		In particular, it holds that
		\[
		\sum_{n=1}^\infty \psi(e^{-\gamma n})e^{\gamma n}<\infty.
		\]
	\end{lemma}
	\begin{proof}
		Combining the monotonicity property of $\psi$ and \cite[Lemma 7]{houzhu25}, we see that
		\[
		\int_0^\infty \psi(f(x)) e^{\gamma x}\mathrm{d}x\leq  \int_0^\infty \psi(e^{-\gamma x}) e^{\gamma x}\mathrm{d}x<\infty,
		\]
		which implies the desired result.
	\end{proof}
	Next, we are going to give a sharp upper bound of $u(x)$. For each $|\omega|=n$ and $0\leq j\leq n-1$, we use $\omega_j$ to denote the label of the ancestor of $\omega$ at generation $j$.
	In particular,
	$\omega_n:= \omega$.
	We have the following useful many-to-one formula (see \cite[p.5]{shi2015}):
	for any $n\in \mathbb{N}$ and any non-negative bounded function $f:\mathbb{R}^n \to [0,\infty)$, it holds that
	\begin{align}\label{many-to-one-formula}
		\mathbb{E}\Bigg(\sum_{|\omega|=n} f(V(\omega_j), 1\leq j\leq n) \Bigg)= m^n\mathbb{E}_0\left(f(S_j,1\leq j\leq n )\right).
	\end{align}
	Denote by $\mathcal{G}_n,n\geq0$ the natural filtration generated by the branching random walk. According to the many-to-one formula \eqref{many-to-one-formula} and Markov property, it is easy to see that
	$$
	\Bigg(\sum_{|\omega|=n}e^{\gamma V(\omega)}, \mathcal{G}_n, n\geq 0, \mathbb{P}\Bigg)
	$$
	is a non-negative martingale. Using Doob's inequality, we have the following result.
	\begin{remark}
		For $x>0$, it follows that
		\begin{align}\label{5tgt4htyh5y}
			u(x)=\mathbb{P}(M>x)\leq \mathbb{P}\left(\max_{n\geq1}\sum_{|\omega|=n}e^{\gamma V(\omega)}>e^{\gamma x}\right)\leq e^{-\gamma x}.
		\end{align}
		Therefore, Lemma \ref{43tt67i8de3} holds with $f=u$.
	\end{remark}
	In Lemma \ref{lem1}, we show the renewal theorem holds for test function $f\in  {\rm DRI}_+(\mathbb{R})$ with compact support. The following lemma shall use the truncation technique to show that this renewal theorem also holds for some function without compact support.
	\begin{lemma}\label{lem1'}
		Assume  $\mathbb{E}[Xe^{\gamma X}]<\infty$ and $\sum_{k=1}^\infty k(\log k)p_k <\infty$. Suppose that $f$ is a non-negative decreasing function on $\mathbb{R}$ such that for any $z\in \mathbb{R}$,
		$
		f(z)\leq \min\{1, e^{-\gamma z}\} .
		$
		\begin{itemize}
			\item[(i)] If $X$ is non-lattice, then
			\begin{align}\label{eq5}
				& \lim_{x\to+\infty} \sum_{\ell=1}^\infty  \widehat{\mathbb{E}}_{-x}\left[\psi(f(-S_\ell)) e^{-\gamma S_\ell}\ind_{\{\overline{S}_\ell \leq 0\}} \right]=\frac{1}{\widehat{\mathbb{E}}_0[S_1]}
				\widehat{\mathbb{E}}_0\left[\int_{-\infty}^{I_\infty} \psi(f(-z))e^{-\gamma z} \mathrm{d}z \right].
			\end{align}
			\item[(ii)] If $X$ is lattice with span $h$, then \begin{align}\label{eq5'}
				& \lim_{k\to+\infty} \sum_{\ell=1}^\infty  \widehat{\mathbb{E}}_{-kh}\left[\psi(f(-S_\ell)) e^{-\gamma S_\ell}\ind_{\{\overline{S}_\ell\leq 0\}} \right]= \frac{h}{\widehat{\mathbb{E}}_0[S_1]}
				\widehat{\mathbb{E}}_0\left[\sum_{i=-\infty}^{I_\infty/h}\psi(f(-ih))e^{-\gamma ih} \right].
			\end{align}
		\end{itemize}
	\end{lemma}
	\begin{proof}
		We only prove the non-lattice case here.
		For each $N\geq 2$, define
		$$f_N(z):= \psi(f(-z))e^{-\gamma z} \ind_{[-N,0]} (z).$$
		By the monotonicity of $\psi$ (see \eqref{psidaoshu}) and Lemma \ref{lem4} we see that for each $N\geq 2$, $f_N\in {\rm DRI}_+(\mathbb{R})$ and that $f_N$ has compact support. Therefore, by Lemma \ref{lem1},
		\begin{align}\label{step28}
			&\lim_{x\to+\infty}  \sum_{\ell=1}^\infty  \widehat{\mathbb{E}}_{-x}\left[f_N(S_\ell)\ind_{\{\overline{S}_\ell \leq 0\}} \right] = \frac{1}{\widehat{\mathbb{E}}_0[S_1]}
			\widehat{\mathbb{E}}_0\left[\int_{-\infty}^{I_\infty} f_N(z)\mathrm{d}z\right].
		\end{align}
		Now for each $N\geq 2$, by \eqref{step28},
		\begin{align}\label{step2}
			&\limsup_{x\to +\infty} \left|\sum_{\ell=1}^\infty  \widehat{\mathbb{E}}_{-x}\left[\psi(f(-S_\ell)) e^{-\gamma S_\ell}\ind_{\{\overline{S}_\ell \leq 0\}} \right] - \frac{1}{\widehat{\mathbb{E}}_0[S_1]}
			\widehat{\mathbb{E}}_0\left[\int_{-\infty}^{I_\infty} \psi(f(-z))e^{-\gamma z} \mathrm{d}z \right]
			\right| \nonumber\\
			&\leq \limsup_{x\to +\infty} \sum_{\ell=1}^\infty  \widehat{\mathbb{E}}_{-x}\left[\psi(f(-S_\ell)) e^{-\gamma S_\ell}\ind_{\{\overline{S}_\ell \leq 0\}} \ind_{\{S_\ell < -N\}}\right] \nonumber\\
			& \qquad + \lim_{x\to +\infty} \left|  \sum_{\ell=1}^\infty  \widehat{\mathbb{E}}_{-x}\left[f_N(S_n)\ind_{\{\overline{S}_\ell\leq 0\}} \right] - \frac{1}{\widehat{\mathbb{E}}_0[S_1]}
			\widehat{\mathbb{E}}_0\left[\int_{-\infty}^{I_\infty} f_N(z)\mathrm{d}z\right]
			\right| \nonumber\\
			&\qquad + \frac{1}{\widehat{\mathbb{E}}_0[S_1]}
			\widehat{\mathbb{E}}_0\left[ \int_{-\infty}^{I_\infty} \psi(f(-z)) e^{-\gamma z}\ind_{\{z<-N\}}\mathrm{d}z \right]\nonumber\\
			& \leq \limsup_{x\to +\infty} \sum_{\ell=1}^\infty  \widehat{\mathbb{E}}_{-x}\left[\psi(e^{\gamma S_\ell}) e^{-\gamma S_\ell} \ind_{\{S_\ell < -N\}}\right] + \frac{1}{\widehat{\mathbb{E}}_0[S_1]} \int_{-\infty}^{-N} \psi(f(-z)) e^{-\gamma z}\mathrm{d}z.
		\end{align}
		According to the definition of $U$, $$\sum_{n=0}^\infty \widehat{\mathbb{P}}_{-x}(S_n\in (k-1, k]) = U(k+x)- U(k+x-1)\leq U(1).$$ Therefore, according to the monotonicity of $\psi$,  for each $N\geq 2$,
		\begin{align}\label{4tt0or4t5q}
			\sum^{\infty}_{\ell=1}\widehat{\mathbb{E}}_{-x}\left[\psi(e^{\gamma S_\ell})e^{-\gamma S_\ell}\ind_{\{S_\ell <-N\}}\right] &\leq \sum_{k=-\infty}^{-N} \sum^{\infty}_{\ell=1}\psi(e^{\gamma k})e^{-\gamma (k-1)} \widehat{\mathbb{P}}_{-x}\left(S_\ell\in (k-1, k]\right)\nonumber\\
			& = e^{\gamma} \sum_{k=N}^{\infty} \psi(e^{-\gamma k})e^{\gamma k} \sum^{\infty}_{\ell=1} \widehat{\mathbb{P}}_{-x}\left(S_\ell\in (-k-1, -k]\right)\nonumber\\
			&\leq U(1)e^{\gamma}\sum_{k=N}^{\infty} \psi(e^{-\gamma k})e^{\gamma k}.
		\end{align}
		Now taking $N\to +\infty $ in  \eqref{step2},  combining Lemma \ref{43tt67i8de3}  and \eqref{4tt0or4t5q}, we get that
		\begin{align}
			&\limsup_{x\to +\infty} \left|\sum_{\ell=1}^\infty  \widehat{\mathbb{E}}_{-x}\left[\psi(f(-S_\ell)) e^{-\gamma S_\ell}\ind_{\{\overline{S}_\ell \leq 0\}} \right] -
			\widehat{\mathbb{E}}_0\left[\int_{-\infty}^{I_\infty} \psi(f(-z))e^{-\gamma z} \mathrm{d}z \right]
			\right| \nonumber\\
			& \stackrel{\eqref{step2}}{\leq} \limsup_{N\to+\infty} \limsup_{x\to +\infty} \sum_{\ell=1}^\infty  \widehat{\mathbb{E}}_{-x}\left[\psi(e^{\gamma S_\ell}) e^{-\gamma S_\ell} \ind_{\{S_\ell < -N\}}\right] + \lim_{N\to+\infty} \frac{1}{\widehat{\mathbb{E}}_0[S_1]} \int_{-\infty}^{-N} \psi(f(-z)) e^{-\gamma z}\mathrm{d}z\nonumber\\
			& \stackrel{\eqref{4tt0or4t5q}}{\leq}  \lim_{N\to+\infty} U(1)e^{\gamma}\sum_{k=N}^{\infty} \psi(e^{-\gamma k})e^{\gamma k}  +\lim_{N\to+\infty} \frac{1}{\widehat{\mathbb{E}}_0[S_1]} \int_{-\infty}^{-N} \psi(u(-z)) e^{-\gamma z}\mathrm{d}z=0,
		\end{align}
		which implies	\eqref{eq5}.
	\end{proof}
	
	Now, we are ready to prove Theorem \ref{456hy5t1}. The sketch of the proof is as follows. We first use the iteration technique to show that $e^{\gamma x}u(x)$ equals to a functional of the random walk $(S_n,n\geq0,\widehat{\mathbb{P}}_{-x}).$ Then, by the renewal theorem proved in
	Lemmas \ref{lem1} and \ref{lem1'},
	we show this functional converges to some constant as $x\to+\infty$. Finally, it suffices to show the limit constant is non-zero. We prove it  with the help of Corollary \ref{Cor:tightness-overshoot}.
	\begin{proof}[Proof of Theorem \ref{456hy5t1}]
		We only prove the case that $X$ is non-lattice.
		Noticing that, in our setting of the branching random walk, the original particle jumps first and then branches.
		Moreover, as mentioned in \cite[proof of Lemma 4.1]{zheng17ptrf}, if $Z_1(\mathbb{R})=0$, then $M=0$.
		Let $ \{M^{(n)}: n\in \mathbb{N}\}$ be i.i.d. copy of $M$.
		So,
		by the Markov property of the branching random walk, it follows that for $x\geq 0$,
		\begin{align}\label{4t5t5gyj}
			1-u(x)&=\mathbb{P}(M\leq x)\cr
			&=	p_0	+ \sum^{\infty}_{n=1}p_n\mathbb{P}(X\leq x,X+M^{(1)}\leq x,X+M^{(2)}\leq x,...,+M^{(n)}\leq x)\cr
			&=		p_0
			+\mathbb{E}\Big(\ind_{\{X\leq x\}}\sum^{\infty}_{n=1}p_n(1-u(x-X))^n\Big)\cr
			&=	p_0+ \mathbb{E}\left(\left(\psi(u(x-X)) +1-p_0-m u(x-X)\right)\ind_{\{X\leq x\}}\right),
		\end{align}
		where the last equality follows from \eqref{4gr6hytjy5}.
		Therefore, it holds that for any $x\geq 0$,
		\begin{align}\label{step3}
			u(x) &=
			1 -p_0-  \mathbb{E}\left(\left(\psi(u(x-X)) +1-p_0-m u(x-X)\right)\ind_{\{X\leq x\}}\right)\nonumber\\
			& =	(1-p_0)\mathbb{P}(X>  x)
			+  m \mathbb{E}\left(u(x-X) \ind_{\{X\leq x\}}\right) -   \mathbb{E}\left(\psi(u(x-X)) \ind_{\{X\leq x\}}\right).
		\end{align}
		Recall that $\{X_i\}_{i\geq1}$ are i.i.d. copies of the step size $X$
		and that $S_n=\sum_{i=1}^nX_i$ under $\mathbb{P}_0$.
		According to iteration,  we see that for any $N\geq 2$,
		\begin{align}
			&u(x) =	(1-p_0)\mathbb{P}_0(S_1>  x)
			-   \mathbb{E}_0\left(\psi(u(x-S_1)) \ind_{\{S_1\leq x\}}\right) \nonumber\\
			&\qquad  +  m \mathbb{E}_0\left( \ind_{\{S_1\leq x\}}\left(	(1-p_0)\mathbb{P}(X>  z)
			-   \mathbb{E}_0 \left[\psi(u(z-X)) \ind_{\{X\leq z\}}\right]\big|_{z=x-S_1}\right)\right) \nonumber\\
			&\qquad  +  m^2  \mathbb{E}_0\left( \ind_{\{S_1\leq x\}}\left(  \mathbb{E}\left(u(z-X) \ind_{\{X\leq z\}}\right) \big|_{z=x-S_1}\right)\right) \nonumber\\
			&= \cdots =(1-p_0) \sum_{\ell=1}^N
			m^{\ell-1}\mathbb{P}_0\left(
			\overline{S}_{\ell-1}
			\leq x< S_\ell\right) - \sum_{\ell=1}^N m^{\ell-1} \mathbb{E}_0\left(\psi(u(x-S_\ell)) \ind_{\{\overline{S}_\ell
				\leq x\}} \right)\nonumber\\
			&\qquad + m^{N} \mathbb{E}_0\left( u(x-S_N) \ind_{\{	\overline{S}_N
				\leq x\}} \right).
		\end{align}
		Since $m\in (0,1)$ and $u\leq 1$, taking $N\to+\infty$ in the above equation yields that for any $x>0$,
		\begin{align}
			&u(x) =	(1-p_0)\sum_{\ell=1}^\infty  m^{\ell-1}\mathbb{P}_0\left(\overline{S}_{\ell-1} \leq x<S_\ell \right)  - \sum_{\ell=1}^\infty m^{\ell-1} \mathbb{E}_0\left(\psi(u(x-S_\ell)) \ind_{\{\overline{S}_\ell \leq x\}} \right).
		\end{align}
		Recall the change-of measure in \eqref{5hyth6y}, we conclude from the above inequality that
		\begin{align}
			&e^{\gamma x}u(x) =	\frac{1-p_0}{m}
			\sum_{\ell=1}^\infty  \widehat{\mathbb{E}}_{-x}\left(e^{-\gamma S_\ell } \ind_{\{\overline{S}_{\ell-1}\leq 0<S_\ell  \}}\right)  - \frac{1}{m}\sum_{\ell=1}^\infty  \widehat{\mathbb{E}}_{-x}\left(\psi(u(-S_\ell)) e^{-\gamma S_\ell}\ind_{\{\overline{S}_\ell \leq 0\}} \right).
		\end{align}
		Therefore, it follows from
		Lemmas \ref{lem1} and \ref{lem1'} that
		\begin{align}\label{step5}
			& \lim_{x\to +\infty} e^{\gamma x}u(x) \nonumber\\
			&=
			(1-p_0)\frac{1-\widehat{\mathbb{E}}_0[e^{-\gamma \mathcal{H}_1 }]}{m\gamma \widehat{\mathbb{E}}_0[\mathcal{H}_1]}
			-  \frac{1}{m\widehat{\mathbb{E}}_0[S_1]}
			\widehat{\mathbb{E}}_0\left[\int_{-\infty}^{I_\infty} \psi(u(-z))e^{-\gamma z} \mathrm{d}z \right]\cr
			&=
			(1-p_0)\frac{1-\widehat{\mathbb{E}}_0[e^{-\gamma \mathcal{H}_1 }]}{m\gamma \widehat{\mathbb{E}}_0[\mathcal{H}_1]}
			-  \frac{1}{m\widehat{\mathbb{E}}_0[S_1]} \widehat{\mathbb E}_0\left[\int_{-I_{\infty}}^{\infty}  \psi(u(z))e^{\gamma z}\mathrm{d}z\right]
			\in [0,1].
		\end{align}

		\par
		Next, we are going to show that the limit is strictly positive.  Define
		\begin{align}\label{Hdef}
			H(s):= \frac{\psi(s)}{ms},~s\in(0,1]
		\end{align}
		and $H(0):=\lim_{s\to0+}H(s).$
		According to direct calculation, we see that
		\begin{align}
			H(s)&=\frac{1}{ms}\left(ms-\sum^{\infty}_{k=0}(1-(1-s)^k)p_k\right) =\frac{1}{m}\sum^{\infty}_{l=0}\Big(\sum^{\infty}_{k=l+1}p_k\Big)(1-(1-s)^l).
		\end{align}
		Therefore, $H$ is	an increasing function.
		Following the same argument as \cite[Lemma 4.5]{zheng17ptrf}	(with $y=0$),
		we have that for any $x \geq 0$ and any $n\in \mathbb{N}$,
		\begin{align}
			u(x)=\mathbb{E}_{0}\left[m^{T_{x}^+\land n} u\left(x-S_{T_x^+\land n}\right) \prod_{j=1}^{T_x^{+}\land n}\Big(1-H\left(u(x-S_j)\right)\Big)\right],
		\end{align}
		where we define the empty product equals $1$.
		Together with \eqref{5hyth6y}, we conclude from the above equality that
		\begin{align}\label{step36}
			u(x) &=\widehat{\mathbb{E}}_{0}\left[e^{-\gamma S_{T_x^+\land n}}u\left(x-S_{T_x^+\land n}\right) \prod_{j=1}^{T_x^{+}\land n}\Big(1-H\left(u(x-S_j)\right)\Big)\right].
		\end{align}
		Since $0\leq H(s)\leq H(1)=1-\frac{1-p_0}{m}	<1$, we have
		\begin{align}
			e^{-\gamma S_{T_x^+\land n}}u\left(x-S_{T_x^+\land n}\right) \prod_{j=1}^{T_x^{+}\land n}\left(1-H\left(u(x-S_j)\right)\right)& \leq e^{-\gamma S_{T_x^+\land n}}u\left(x-S_{T_x^+\land n}\right)  \leq e^{-\gamma x},
		\end{align}
		where the last inequality follows from \eqref{5tgt4htyh5y} (observe that \eqref{5tgt4htyh5y} holds for $x\leq0$ trivially). Thus, by the dominated convergence theorem, we can take $n\to+\infty$ in \eqref{step36} and get that
		\begin{align}\label{step37}
			e^{\gamma x} u(x)& =e^{\gamma x} \widehat{\mathbb{E}}_{0}\left[e^{-\gamma S_{T_x^+}}u\left(x-S_{T_x^+}\right) \prod_{j=1}^{T_x^{+}}\left(1-H\left(u(x-S_j)\right)\right)\right]\nonumber\\
			& = \widehat{\mathbb{E}}_{-x}\left[e^{-\gamma S_{T_0^+}}u\left(-S_{T_0^+}\right) \prod_{j=1}^{T_0^{+}}\left(1-H\left(u(-S_j)\right)\right)\right]\nonumber\\
			& = \widehat{\mathbb{E}}_{-x}\left[e^{-\gamma S_{T_0^+}} \prod_{j=1}^{T_0^{+}}\left(1-H\left(u(-S_j)\right)\right)\right],
		\end{align}
		where in the last equality we used the fact that $u(z)=1$ for all $z<0$. Combining \eqref{step37}
		and inequality $\log(1-y)\geq {-cy}$ for all $y\in
		[0, H(1)]$
		for some constant $c>0$, it holds that
		\begin{align}\label{step39}
			e^{\gamma x} u(x)&\geq \widehat{\mathbb{E}}_{-x}\left[
			\exp\left\{ -\gamma S_{T_0^+} - c \sum_{j=1}^{T_0^+} H\left(u(-S_j)\right)\right\}
			\right].
		\end{align}
		{Noticing that \eqref{4tt0or4t5q} is still valid with $N$ replaced by $0$}, combining the monotonicity property of $H$ and
		\eqref{5tgt4htyh5y},
		we get that
		\begin{align}\label{step38}
			\widehat{\mathbb{E}}_{-x}  \left[\sum_{j=1}^{T_0^+} H\left(u(-S_j)\right)\right] & \leq H(1)+ \widehat{\mathbb{E}}_{-x}  \left[\sum_{j=1}^{T_0^+-1} H\left(e^{\gamma S_j}\right)\right]\nonumber\\
			& \leq H(1)+ \frac{1}{m}\widehat{\mathbb{E}}_{-x}  \left[\sum_{j=1}^{\infty } { \psi}\left(e^{\gamma S_j}\right)e^{-\gamma S_j}\ind_{\{S_j\leq 0\}}\right]\nonumber\\
			&\leq H(1)+ \frac{1}{m}U(1)e^{\gamma}\sum_{k=0}^{\infty} \psi(e^{-\gamma k})e^{\gamma k},
		\end{align}
		where the second inequality follows from the definition \eqref{Hdef}.
		Combining Corollary \ref{Cor:tightness-overshoot} and \eqref{step38}, we see that
		\begin{align}
			& \limsup_{K\to+\infty} \sup_{x>0} \widehat{\mathbb{P}}_{-x}\left(\gamma S_{T_0^+} + c \sum_{j=1}^{T_0^+} H\left(u(-S_j)\right) >K \right) \nonumber\\
			& \leq \lim_{K\to+\infty} \sup_{x>0} \widehat{\mathbb{P}}_{-x}\left(\gamma S_{T_0^+}>K/2\right) + \limsup_{K\to+\infty} \frac{2c}{K} \sup_{x>0} 	\widehat{\mathbb{E}}_{-x}  \left[\sum_{j=1}^{T_0^+} H\left(u(-S_j)\right)\right]  = 0.
		\end{align}
		Therefore, there exists $K_0>0$ such that
		\[
		\sup_{x>0} \widehat{\mathbb{P}}_{-x}\left(\gamma S_{T_0^+} + c \sum_{j=1}^{T_0^+} H\left(u(-S_j)\right) >K_0 \right)\leq \frac{1}{2}.
		\]
		Plugging this back to \eqref{step39} yields that
		\[
		\lim_{x\to+\infty}e^{\gamma x}u(x)\geq \liminf_{x\to+\infty}e^{-K_0}\widehat{\mathbb{P}}_{-x}\left(\gamma S_{T_0^+} + c \sum_{j=1}^{T_0^+} H\left(u(-S_j)\right) \leq K_0 \right)\geq \frac{1}{2}e^{-K_0}>0,
		\]
		as desired.
	\end{proof}

	\section{Proof of Theorem  \ref{theom3}: Tail probability of $M^{(0,\infty)}$}
	In this section, we shall study the maximal displacement of branching random walk with killing. To simplify the notation, let $u(y,x):=\mathbb{P}_{\delta_y}(M^{(0,\infty)}>x)$ for $y\in(0,x)$. Note that $u(y,x)=1$ for $y\in[x,\infty).$
	\par
	We first give $u(y,x)$ a representation in terms of the random walk $\{S_n\}_{n\geq0}$; see \eqref{step13} below. For $y\in(0,x]$, observe that
	there are two killing mechanisms in the first generation that cause $M^{(0,\infty)}=0$: the first one is that the original particle
	stays in
	$(-\infty, 0]$; the second one is that the original particle jumps into $(0,\infty)$ but $Z_1(\mathbb{R})=0$. Therefore,
	\begin{align}
		1-u(y,x)&=\mathbb{P}_{\delta_y}
		(M^{(0,\infty)}\leq x)\cr
		&=
		\mathbb{P}(y+X\leq 0)+\mathbb{P}(y+X>0)p_0
		+\sum^{\infty}_{n=1}p_n\mathbb{E}\big(\mathbb{P}_{\delta_{y+X}}
		(M^{(0,\infty)}\leq x)^n\ind_{\{y+X\in (0,x]\}}\big)\cr
		&=
		p_0+ (1-p_0)\mathbb{P}(X\leq -y)+ \sum^{\infty}_{n=1}p_n\mathbb{E}\big((1-u(y+X,x))^n\ind_{\{y+X\in(0,x]\}}\big)\cr
		&=
		p_0+(1-p_0)\mathbb{P}(X\leq -y)+ \mathbb{E}\left(\left(\psi(u(y+X,x))+1-p_0-mu(y+X,x) \right) \ind_{\{y+X\in(0,x]\}}\right).
	\end{align}
	Thus, for any $y\in (0,x]$,
	\begin{align}
		&u(y,x) =\cr
		&(1-p_0)\mathbb{P}(X>x-y)
		+ m \mathbb{E}\left(u(y+X,x) \ind_{\{y+X\in(0,x]\}}\right)- \mathbb{E}\left(\psi(u(y+X,x)) \ind_{\{y+X\in(0,x]\}}\right).
	\end{align}
	For any $N\geq 2$, iterating the above equation $N$-times, we get that
	\begin{align}
		u(y,x) & =	(1-p_0) \mathbb{P}_0(S_1>x-y)
		-\mathbb{E}_0\left(\psi(u(y+S_1,x)) \ind_{\{y+S_1\in(0,x]\}}\right)\nonumber\\
		&\qquad +	m(1-p_0) \mathbb{P}_0(y+S_1\in (0, x], y+S_2>x)
		-m
		\mathbb{E}_0\left(\psi(u(y+S_2,x)) \ind_{\{y+S_1, y+S_2\in(0,x]\}}\right) \nonumber\\
		& \qquad +m^2 \mathbb{E}_0\left( u(y+S_2,x)\ind_{\{y+S_1, y+S_2\in(0,x]\}}\right)\nonumber\\
		&=\cdots =	(1-p_0)\sum_{\ell=1}^N
		m^{\ell-1}  \mathbb{P}_0\big(y+S_j \in (0, x], 1\leq j\leq \ell-1 , y+S_\ell >x\big) \nonumber\\
		&\qquad - \sum_{\ell=1}^N  m^{\ell-1}\mathbb{E}_0\left(\psi(u(y+S_\ell,x)) \ind_{\{y+S_j\in(0,x], 1\leq j\leq \ell \}}\right)\nonumber\\
		&\qquad + m^{N} \mathbb{E}_0\left( u(y+S_N,x)\ind_{\{y+S_j \in(0,x], 1\leq j\leq N\}}\right).
	\end{align}
	Noticing that both of $u$ and $\psi(u)$ are bounded, taking $N\to\infty$, we conclude that
	\begin{align}
		u(y,x) & =	(1-p_0)\sum_{\ell=1}^\infty
		m^{\ell-1}  \mathbb{P}_0\Big(y+S_j \in (0, x], 1\leq j\leq \ell-1 , y+S_\ell >x\Big) \nonumber\\
		&\qquad - \sum_{\ell=1}^\infty   m^{\ell-1}\mathbb{E}_0\left(\psi(u(y+S_\ell,x)) \ind_{\{y+S_j\in(0,x], 1\leq j\leq \ell \}}\right).
	\end{align}
	By \eqref{5hyth6y}, the above equation is equivalent to
	\begin{align}\label{step13}
		e^{\gamma(x-y)}u(y,x) & =
		\frac{e^{\gamma(x-y)}(1-p_0)}{m}
		\sum_{\ell=1}^\infty    \widehat{\mathbb{E}}_0( e^{-\gamma S_\ell}\ind_{\{y+S_j \in (0, x], 1\leq j\leq \ell-1 , y+S_\ell >x \}} ) \nonumber\\
		&\qquad - \frac{e^{\gamma(x-y)}}{m}\sum_{\ell=1}^\infty   \widehat{\mathbb{E}}_0\left(\psi(u(y+S_\ell,x)) e^{-\gamma S_\ell}\ind_{\{y+S_j\in(0,x], 1\leq j\leq \ell \}}\right)\nonumber\\
		& =	\frac{1-p_0}{m}
		\sum_{\ell=1}^\infty    \widehat{\mathbb{E}}_{y-x}\left( e^{-\gamma S_\ell}\ind_{\{S_j \in (-x, 0], 1\leq j\leq \ell-1 , S_\ell >0\}} \right) \nonumber\\
		&\qquad - \frac{1}{m}\sum_{\ell=1}^\infty   \widehat{\mathbb{E}}_{y-x}\left(\psi(u(x+S_\ell,x)) e^{-\gamma S_\ell}\ind_{\{S_j\in(-x,0], 1\leq j\leq \ell \}}\right).
	\end{align}
	\par
	From above, one can see that to prove Theorem \ref{theom3}, it suffices to study limit behaviours of the two terms on the r.h.s. of \eqref{step13} as $x\to +\infty$. To this end, the following lemma shows the convergence of the second term.
	
	\begin{lemma}\label{lem3'}
		Assume  $\mathbb{E}[X e^{\gamma X}]<\infty $ and $\sum_{k=1}^\infty k(\log k)p_k<\infty$.
		\begin{itemize}
			\item[(i)]If $X$ is non-lattice,	then for any $y>0$,
			\begin{align}\label{eq6}
				& \lim_{x\to+\infty} \sum_{\ell=1}^\infty   \widehat{\mathbb{E}}_{y-x}\left(\psi(u(x+S_\ell,x)) e^{-\gamma S_\ell}\ind_{\{S_j\in(-x,0], 1\leq j\leq \ell \}}\right) \nonumber\\
				& =\frac{\widehat{\mathbb{P}}_{0}\left(I_\infty > -y \right) }{\widehat{\mathbb{E}}_0[S_1]}
				\widehat{\mathbb{E}}_0\left[ \int_{-\infty}^{I_\infty}  \psi(u(-z))e^{-\gamma z}\mathrm{d}z \right].
			\end{align}
			\item[(ii)] 	If $X$ is lattice with span $h>0$,
			then for any $i\geq1$,
			\begin{align}\label{eq6'}
				& \lim_{k\to+\infty} \sum_{\ell=1}^\infty   \widehat{\mathbb{E}}_{(i-k)h}\left(\psi(u(kh+S_\ell,kh)) e^{-\gamma S_\ell}\ind_{\{S_j\in(-kh,0], 1\leq j\leq \ell \}}\right) \nonumber\\
				& =\frac{h \widehat{\mathbb{P}}_{0}\left(I_\infty>
					-ih	\right) }{\widehat{\mathbb{E}}_0[S_1]}
				\widehat{\mathbb{E}}_0\left[ \sum_{j=-\infty}^{I_\infty/h}  \psi(u(-jh))e^{\gamma jh} \right] .
			\end{align}
		\end{itemize}
	\end{lemma}
	\begin{proof}
		We only prove \eqref{eq6} here since the proof of \eqref{eq6'} is similar. For $N\geq1$, define
		\begin{align}
			Q_{1,N}(x)& := \sum_{\ell=1}^N    \widehat{\mathbb{E}}_{y-x}\left(\psi(u(x+S_\ell,x)) e^{-\gamma S_\ell}\ind_{\{S_j\in(-x,0], 1\leq j\leq \ell \}}\right),\nonumber\\
			Q_{2,N}(x)&:=  \sum_{\ell=N+1}^\infty     \widehat{\mathbb{E}}_{y-x}\left(\psi(u(x+S_\ell,x)) e^{-\gamma S_\ell}\ind_{\{S_j\in(-x,0], 1\leq j\leq \ell \}}\right).
		\end{align}
		Then,
		\begin{align}\label{step32}
			\sum_{\ell=1}^\infty   \widehat{\mathbb{E}}_{y-x}\left(\psi(u(x+S_\ell,x)) e^{-\gamma S_\ell}\ind_{\{S_j\in(-x,0], 1\leq j\leq \ell \}}\right) = Q_{1,N}(x)+ Q_{2,N}(x).
		\end{align}
		For $Q_{1,N}(x)$, combining
		\eqref{5tgt4htyh5y},
		\begin{align}\label{4tyhy7jude3}
			u(y,x)=\mathbb{P}_{\delta_y}\left(M^{(0,\infty)}>x\right)\leq \mathbb{P}_{\delta_y}\left(M>x\right)= u(x-y)
		\end{align}
		and that $\psi$ is increasing, we see that for any
		$z\leq 0$,
		\begin{align}\label{bounds-for-u}
			\psi(u(x+z, x))		e^{-\gamma z}
			\leq \psi(u(-z))e^{-\gamma z}\leq\psi(e^{\gamma z}) e^{-\gamma z} \leq \psi(e^{\gamma (\lfloor z\rfloor +1)}) e^{-\gamma \lfloor z\rfloor}.
		\end{align}
		By Lemma \ref{43tt67i8de3}, above yields that there exists some constant $K>0$ such that for all  $x>0\geq z$,
		$$\psi(u(x+z, x))e^{-\gamma z}\leq K.$$
		Therefore, for each $N\geq 2$,
		\begin{align}\label{step20'}
			&  \limsup_{x\to+\infty} Q_{1,N}(x)=
			\limsup_{x\to+\infty} \sum_{\ell=1}^N    \widehat{\mathbb{E}}_{y-x}\left[\psi(u(x+S_\ell,x)) e^{-\gamma S_\ell}\ind_{\{S_j\in(-x,0], 1\leq j\leq \ell \}}\right]\nonumber\\
			& \leq \limsup_{x\to+\infty} \sum_{\ell=1}^\infty    \widehat{\mathbb{E}}_{y-x}\left[\psi(e^{\gamma S_\ell})e^{-\gamma S_\ell}\ind_{\{S_\ell  \leq -\lfloor x/2 \}}\right] +K \lim_{x\to+\infty} \sum_{\ell=1}^N    \widehat{\mathbb{P}}_{y-x}\left( S_\ell  >-\lfloor x/2\rfloor  \right)  \nonumber\\
			&= \limsup_{x\to+\infty}  \sum_{\ell=1}^\infty    \widehat{\mathbb{E}}_{y-x}\left[\psi(e^{\gamma S_\ell})e^{-\gamma S_\ell}\ind_{\{S_\ell  \leq -\lfloor x/2\rfloor \}}\right] ,
		\end{align}
		where the last equality follows from
		$$\lim_{x\to+\infty}  \widehat{\mathbb{P}}_{y-x}\left( S_\ell  >-\lfloor x/2\rfloor  \right) =\lim_{x\to+\infty}  \widehat{\mathbb{P}}_{0}\left(S_\ell  >-y+x- \lfloor x/2\rfloor\right)  = 0$$
		for all $1\leq \ell \leq N$.
		Taking $N= \lfloor x/2 \rfloor$ in \eqref{4tt0or4t5q}
		and plugging this into \eqref{step20'}, we obtain that
		\begin{align}\label{step20}
			\limsup_{x\to+\infty} Q_{1,N}(x)
			& \leq \lim_{x\to+\infty} U(1)e^{\gamma }\sum_{n=\lfloor x/2 \rfloor}^\infty \psi(e^{-\gamma n})e^{\gamma n} =0.
		\end{align}

			Now we treat $Q_{2, N}(x)$.
			We define
			\begin{align}
				Q_{3,N}(x)&:= \sum_{\ell=N+1}^\infty
				\widehat{\mathbb{E}}_{y-x}\Big(  \psi(u(x+S_\ell,x)) e^{-\gamma S_\ell} \ind_{\{ \underline{S}_N > -x, \overline{S}_\ell  \leq 0 \}} \Big) \nonumber\\
				Q_{4,N}(x)&:=   \sum_{\ell=N+1}^\infty    \widehat{\mathbb{E}}_{y-x}\Big(  \psi(u(x+S_\ell,x)) e^{-\gamma S_\ell}  \ind_{\{ \underline{S}_\ell > -x, \max_{N+1\leq  j\leq \ell} S_j \leq  0 \}} \Big).
			\end{align}
			Similar to \eqref{step14}, combining Lemma \ref{43tt67i8de3} and \eqref{bounds-for-u}, we have that
			\begin{align}\label{step21}
				& |Q_{2,N}(x)- Q_{3,N}(x)|\nonumber\\
				& \leq  \sum_{\ell=N+1}^\infty     \widehat{\mathbb{E}}_{y-x}\left( \psi(e^{\gamma S_\ell}) e^{-\gamma S_\ell} \ind_{\{ T_{-x}^-\in [N+1, \ell], S_\ell \leq 0 \}} \right) \nonumber\\
				&\leq  \sum_{k=1}^\infty  \psi(e^{-\gamma (k-1)})e^{\gamma k}\sum_{\ell=N+1}^\infty     \widehat{\mathbb{E}}_{y-x}\left(  T_{-x}^-\in [N+1, \ell], S_\ell \in (-k, 1-k] \right) \nonumber\\
				& \leq  U(1)e^{\gamma} \sum_{k=0}^\infty  \psi(e^{-\gamma k})e^{\gamma k} \widehat{\mathbb{P}}_y(T_0^- \in (N,+\infty)).
			\end{align}
			Moreover, repeating the same argument as that leading to \eqref{step15}, we also get that
			\begin{align}\label{step22}
				& |Q_{3,N}(x)-Q_{4,N}(x)| \nonumber\\
				& \leq  \sum_{\ell=N+1}^\infty  \widehat{\mathbb{E}}_{y-x} \left(\psi(e^{\gamma S_\ell})e^{-\gamma S_\ell} \ind_{\{ \overline{S}_N >0, S_\ell <0 \}}\right)\nonumber\\
				& \leq  \sum_{k=1}^\infty \psi(e^{-\gamma (k-1)})e^{\gamma k} \widehat{\mathbb{E}}_{y-x} \left( \ind_{\{ \overline{S}_N >0 \}}  \sum_{\ell=1}^\infty \widehat{\mathbb{P}}_{z}( S_\ell \in (-k, 1-k])\big|_{z=S_N} \right) \nonumber\\
				&\leq U(1)e^\gamma  \sum_{k=0}^\infty \psi(e^{-\gamma k})e^{\gamma k}  \widehat{\mathbb{P}}_{y-x}(\overline{S}_N  >0).
			\end{align}
			Combining \eqref{step21} and \eqref{step22}, we conclude that
			\begin{align}\label{step23}
				\limsup_{x\to+\infty} |Q_{2,N}(x)-Q_{4,N}(x)|
				& \leq
				U(1)e^{\gamma} \sum_{k=0}^\infty  \psi(e^{-\gamma k})e^{\gamma k} \widehat{\mathbb{P}}_y(T_0^- \in (N,+\infty)).
			\end{align}
			Finally, since $u(y,x)\leq u(x-y)$ and that $\psi$ is increasing, according to the Markov property,
			\begin{align}
				&\limsup_{x\to+\infty} Q_{4,N}(x)=
				\limsup_{x\to+\infty}\sum_{\ell=N+1}^\infty \widehat{\mathbb{E}}_{y-x}\Big(  \psi(u(x+S_\ell,x)) e^{-\gamma S_\ell}  \ind_{\{ \underline{S}_N > -x, \max_{N+1\leq  j\leq \ell} S_j  \leq  0 \}} \Big) \nonumber\\
				& \leq \limsup_{x\to+\infty}\sum_{\ell=N+1}^\infty \widehat{\mathbb{E}}_{y-x}\Big(  \psi(u(-S_\ell)) e^{-\gamma S_\ell}  \ind_{\{ \underline{S}_N > -x, \max_{N+1\leq  j\leq \ell} S_j \leq  0 \}}
				\Big) \nonumber\\
				& = \lim_{x\to+\infty} \widehat{\mathbb{E}}_{y}\left( \ind_{\{ \underline{S}_N > 0 \}}  \sum_{\ell=1}^\infty
				\widehat{\mathbb{E}}_{z-x} (\psi(u(-S_\ell)) e^{-\gamma S_\ell}  \ind_{\{  \max_{1\leq  j\leq \ell} S_j \leq  0 \}})\big|_{z=S_N} \right).
			\end{align}
			Noticing that for each $z>0$,
			\begin{align}
				\lim_{x\to+\infty}  \sum_{\ell=1}^\infty
				\widehat{\mathbb{E}}_{z-x}
				\Big(\psi(u(-S_\ell)) e^{-\gamma S_\ell}  \ind_{\{  \max_{1\leq  j\leq \ell} S_j \leq  0 \}}\Big)  = 	\lim_{x\to+\infty}  \sum_{\ell=1}^\infty
				\widehat{\mathbb{E}}_{z-x}
				\Big(\psi(u(-S_\ell)) e^{-\gamma S_\ell}  \ind_{\{  \overline{S}_{\ell} \leq  0 \}}\Big)  .
			\end{align}
			Moreover, by	\eqref{4tt0or4t5q},
			$ \sum_{\ell=1}^\infty
			\widehat{\mathbb{E}}_{z-x}
			\Big(\psi(u(-S_\ell)) e^{-\gamma S_\ell}  \ind_{\{  \max_{1\leq  j\leq \ell} S_j \leq  0 \}}\Big) $ is
			a bounded function.
			Therefore, combining Lemma \ref{lem1'} and the  dominated convergence theorem, we obtain that
			\begin{align}\label{step24}
				& \limsup_{x\to+\infty} Q_{4,N}(x)\nonumber\\
				& \leq  \widehat{\mathbb{E}}_{y}\left( \ind_{\{ \underline{S}_N > 0 \}}  \lim_{x\to+\infty} \sum_{\ell=1}^\infty \widehat{\mathbb{E}}_{S_N-x} \Big(\psi(u(-S_\ell)) e^{-\gamma S_\ell}  \ind_{\{  \overline{S}_\ell   \leq  0 \}}\Big) \right)\nonumber\\
				& =  \frac{\widehat{\mathbb{P}}_{y}\left(
					\underline{S}_N
					> 0\right)}{\widehat{\mathbb{E}}_0[S_1]}
				\widehat{\mathbb{E}}_0\left[\int_{-\infty}^{I_\infty}\psi(u(-z))e^{-\gamma z} \mathrm{d}z \right].
			\end{align}
			Also for the lower bound, noticing that $u(x+y, x)$ is increasing in $x$, for any $K>0$, it holds that
			\begin{align}\label{step25'}
				&\liminf_{x\to+\infty} Q_{4,N}(x)=
				\liminf_{x\to+\infty}\sum_{\ell=N+1}^\infty \widehat{\mathbb{E}}_{y-x}\Big(  \psi(u(x+S_\ell,x)) e^{-\gamma S_\ell}  \ind_{\{ \underline{S}_N > -x, \max_{N+1\leq  j\leq \ell} S_\ell \leq  0 \}} \Big) \nonumber\\
				& \geq \liminf_{x\to+\infty}\sum_{\ell=N+1}^\infty \widehat{\mathbb{E}}_{y-x}\Big(  \psi(u(K+S_\ell,K)) e^{-\gamma S_\ell}  \ind_{\{ \underline{S}_N > -x, \max_{N+1\leq  j\leq \ell} S_\ell \leq  0 \}} \Big) \nonumber\\
				& = \lim_{x\to+\infty} \widehat{\mathbb{E}}_{y}\left( \ind_{\{ \underline{S}_N> 0 \}}  \sum_{\ell =1}^\infty
				\widehat{\mathbb{E}}_{z-x} \Big(\psi(u(K+S_\ell, K)) e^{-\gamma S_\ell}  \ind_{\{  \max_{N+1\leq  j\leq \ell} S_j \leq  0 \}}\Big)\big|_{z=S_N} \right)\nonumber\\
				& = \frac{\widehat{\mathbb{P}}_{y}\left(\underline{S}_N > 0\right)}{\widehat{\mathbb{E}}_0[S_1]}
				\widehat{\mathbb{E}}_0\left[ \int_{-\infty}^{I_\infty}   \psi(u(K+z,K))e^{-\gamma z}\mathrm{d}z\right],
			\end{align}
			where in the last equality we used Lemma \ref{lem1'}. Recall that $\mathbb{T}$ is the branching process with offspring distribution $\{p_k\}_{k\geq0}$.
			For $K,~z>0$, by \eqref{4tyhy7jude3} and
			\begin{align}
				u(K+z, K)& \geq  \mathbb{P}_{\delta_{K+z}}\left(M^{(0,\infty)}> K, \min_{\omega\in \mathbb{T}} V(\omega)>0 \right)= \mathbb{P}_{\delta_{K+z}}\left(M> K, \min_{\omega\in \mathbb{T}} V(\omega)>0 \right)\nonumber\\
				&\geq \mathbb{P}_{\delta_{K+z}}\left(M> K \right)- \mathbb{P}_{\delta_{K+z}}\left( \min_{\omega\in \mathbb{T}} V(\omega)\leq 0 \right)= u(-z)- \mathbb{P}\left( \min_{\omega\in \mathbb{T}} V(\omega)\leq -K-z \right),
			\end{align}
			we have $\lim_{K\to +\infty} u(K+z,K)=u(-z)$.
			Therefore, by Fatou's lemma,
			the left hand side of \eqref{step25'} has lower bound
			\begin{align}\label{step25}
				&
				\liminf_{x\to+\infty} Q_{4,N}(x)\nonumber\\
				&\geq \lim_{K\to +\infty} \frac{\widehat{\mathbb{P}}_{y}\left(\underline{S}_N > 0\right)}{\widehat{\mathbb{E}}_0[S_1]}
				\widehat{\mathbb{E}}_0\left[ \int_{-\infty}^{I_\infty} \psi(u(K+z,K))e^{-\gamma z}\mathrm{d}z\right]\nonumber\\
				& \geq  \frac{\widehat{\mathbb{P}}_{0}\left(\underline{S}_N  > -y\right)}{\widehat{\mathbb{E}}_0[S_1]}
				\widehat{\mathbb{E}}_0\left[\int_{-\infty}^{I_\infty} \psi(u(-z))e^{-\gamma z}\mathrm{d}z \right].
			\end{align}
			Together with \eqref{step24} and \eqref{step25}, we conclude that
			\begin{align}\label{step35}
				\lim_{x\to +\infty} Q_{4,N}(x)
				& =  \frac{\widehat{\mathbb{P}}_{y}\left(
					\underline{S}_N
					> 0\right)}{\widehat{\mathbb{E}}_0[S_1]}
				\widehat{\mathbb{E}}_0\left[\int_{-\infty}^{I_\infty} \psi(u(-z))e^{-\gamma z}\mathrm{d}z \right].
			\end{align}
			Combining the arguments above, we conclude that
			\begin{align}
				& \lim_{x\to+\infty} \sum_{\ell=1}^\infty   \widehat{\mathbb{E}}_{y-x}\left(\psi(u(x+S_\ell,x)) e^{-\gamma S_\ell}\ind_{\{S_j\in(-x,0], 1\leq j\leq \ell \}}\right)
				\stackrel{\eqref{step32}}{=}\lim_{x\to+\infty} \left(Q_{1,N}(x)+Q_{2,N}(x)\right)\nonumber\\
				&\stackrel{\eqref{step20}}{=} \lim_{N\to+\infty} \lim_{x\to+\infty}
				Q_{2,N}(x) \stackrel{\eqref{step23}}{=}\lim_{N\to+\infty} \lim_{x\to+\infty} 	Q_{4,N}(x) \nonumber\\
				& \stackrel{\eqref{step35}}{=}
				\frac{\widehat{\mathbb{P}}_{0}\left( I_\infty > -y\right)}{\widehat{\mathbb{E}}_0[S_1]}
				\widehat{\mathbb{E}}_0\left[\int_{-\infty}^{I_\infty} \psi(u(-z))e^{-\gamma z}\mathrm{d}z \right],
			\end{align}
			which completes the proof of 	\eqref{eq6}.
		\end{proof}

		Now we are ready to prove Theorem \ref{theom3}. Without loss of generality, here we assume the step size is non-lattice.
		\begin{proof}[Proof of Theorem \ref{theom3}]
			Combining Lemma \ref{lem3} and Lemma \ref{lem3'}, we deduce that
			\begin{align}
				\lim_{x\to+\infty} e^{\gamma(x-y)} \mathbb{P}_{\delta_y}(M>x)&=
				\lim_{x\to+\infty} e^{\gamma(x-y)} u(y,x)\nonumber\\
				& \stackrel{\eqref{step13}}{=}	\lim_{x\to+\infty}
				\frac{1-p_0}{m}
				\sum_{\ell=1}^\infty    \widehat{\mathbb{E}}_{y-x}\Big( e^{-\gamma S_\ell}\ind_{\{S_j \in (-x, 0], 1\leq j\leq \ell-1 , S_\ell >0\}} \Big) \nonumber\\
				&\qquad -	\lim_{x\to+\infty} \frac{1}{m}\sum_{\ell=1}^\infty   \widehat{\mathbb{E}}_{y-x}\left[\psi(u(x+S_\ell,x)) e^{-\gamma S_\ell}\ind_{\{S_j\in(-x,0], 1\leq j\leq \ell \}}\right]\nonumber\\
				& = \widehat{\mathbb{P}}_{0}\Big(I_{\infty} > -y\Big) \lim_{x\to+\infty} e^{\gamma x} u(x)		,
			\end{align}
			where the last equality follows from Theorem \ref{456hy5t1}.
		\end{proof}
		
		\section{Proof of Theorems \ref{54th6uyju5t} and \ref{theom4}: Tail probabilities of $M_n$,~$M^{(0,\infty)}_n$ }
		We first prove Theorem \ref{54th6uyju5t}. Recall that $M_n:=\max_{|\omega|=i,0\leq i\leq n}V(\omega)~\text{for}~n\geq0.$ To prove the theorem, we divide proofs into two cases according to $c>\widehat{\mathbb{E}}_0[S_1]$ and $c<\widehat{\mathbb{E}}_0[S_1]$.
		
		\noindent
		\begin{proof}[Proof of Theorem \ref{54th6uyju5t}.] \textbf{Case 1:~$c>\widehat{\mathbb{E}}_0[S_1]$.}
			Set $x_+=\max\{0,x\}.$ For $n\geq0$, define
			\begin{align}\label{Def-of-D-n}
				D_n^+:= \sum_{|u|=n}\left(V(u)- n\widehat{\mathbb{E}}_0[S_1]\right)_+ e^{\gamma V(u)}.
			\end{align}
			Recall that $\mathcal{G}_n,n\geq0$ is the natural filtration of the branching random walk. By the Markov property and many-to-one formula \eqref{many-to-one-formula},
			\begin{align}
				\mathbb{E}\left(D_{n+1}^+ \Big| \mathcal{G}_n \right) &= \sum_{|u|=n} \mathbb{E}_{\delta_{V(u)}}\Bigg[ \sum_{|w|=1} \left(V(w)-(n+1)\widehat{\mathbb{E}}_0[S_1] \right)_+ e^{\gamma V(w)}\Bigg]\nonumber\\
				&\geq \sum_{|u|=n} \left(\mathbb{E}_{\delta_{V(u)}}\Bigg[ \sum_{|w|=1} \left(V(w)-(n+1)\widehat{\mathbb{E}}_0[S_1] \right)  e^{\gamma V(w)}\Bigg]\right)_+\nonumber\\
				&  = \sum_{|u|=n} e^{\gamma V(u)}\Big( \widehat{\mathbb{E}}_{0}\left[  S_1+z-(n+1)\widehat{\mathbb{E}}_0[S_1]  \right]\Big)_+\bigg|_{z=V(u)} \nonumber\\
				& = D_n^+.
			\end{align}
			Therefore, $\{D_n^+,\mathcal{G}_n, n\geq 0, \mathbb{P} \}$ is a non-negative submartingale. Now according to Doob's inequality, for $c>\widehat{\mathbb{E}}_0[S_1]$,
			\begin{align}\label{step46}
				e^{\gamma cn}\mathbb{P}(M_n\geq cn) & = 	e^{\gamma cn}\mathbb{P}\Big(\max_{s\leq n} \max_{|u|=s} V(u)\geq cn\Big) \nonumber\\
				& \leq e^{\gamma cn}\mathbb{P}\left(\max_{s\leq n} \max_{|u|=s} (V(u)- s\widehat{\mathbb{E}}_0[S_1] ) e^{\gamma V(u)}\geq (c-\widehat{\mathbb{E}}_0[S_1])ne^{\gamma cn}\right) \nonumber\\
				&\leq e^{\gamma cn}\mathbb{P} \left(\max_{s\leq n} D_s^+ \geq(c-\widehat{\mathbb{E}}_0[S_1])ne^{\gamma cn} \right) \leq \frac{1}{(c-\widehat{\mathbb{E}}_0[S_1])n} \mathbb{E}\left(D_n^+\right).
			\end{align}
			By many-to-one formula \eqref{many-to-one-formula}, we conclude that
			\begin{align}\label{4frfrfgttr4}
				e^{\gamma cn}\mathbb{P}(M_n\geq cn)
				\leq \frac{1}{(c-\widehat{\mathbb{E}}_0[S_1])n} \widehat{\mathbb{E}}_0 \left((S_n- \widehat{\mathbb{E}}_0[S_1] n)_+\right).
			\end{align}
			According to the ergodic theorem (see the last remark of \cite[p.285]{durrett10}), $\frac{S_n-\widehat{\mathbb{E}}_0[S_1]n}{n}$ converges in $L^1(\widehat{\mathbb{P}}_0)$ to $0$ as $n\to\infty$. This, combined with \eqref{4frfrfgttr4}, yields that for $c>\widehat{\mathbb{E}}_0[S_1]$,
			\begin{align}\label{4rgt5fr4}
				\lim_{n\to\infty}e^{\gamma cn}\mathbb{P}(M_n\geq cn)=0.
			\end{align}
			\par
			If the step size is lattice with span $h$, then by Theorem \ref{456hy5t1} (ii), we have for any $c>0$,
			\begin{align}
				\liminf_{n\to\infty}e^{\gamma cn}\mathbb{P}(M\geq cn)&\geq \lim\inf_{n\to\infty}e^{\gamma cn-\gamma h\lceil cn/h\rceil}e^{\gamma h\lceil cn/h\rceil}\mathbb{P}(M\geq h\lceil cn/h\rceil)\cr
				&\geq e^{-\gamma h} \lim_{n\to\infty}e^{\gamma h\lceil cn/h\rceil}\mathbb{P}(M\geq h\lceil cn/h\rceil)>0,
			\end{align}
			where $\lceil x\rceil$ stands for the smallest integer not less than $x$. This, combined with Theorem \ref{456hy5t1} (i), yields that whether the step size is lattice or not, we always have
			\begin{align}\label{4gh6tt5f}
				\liminf_{n\to\infty}e^{\gamma cn}\mathbb{P}(M\geq cn)>0.
			\end{align}
			Above, combined with \eqref{4rgt5fr4}, entails that for $c>\widehat{\mathbb{E}}_0[S_1]$,
			\begin{align}\label{4gh6y6y55k}
				\lim_{n\to\infty}\mathbb{P}(M_n\geq cn|M>cn)=\lim_{n\to\infty}\frac{e^{\gamma cn}\mathbb{P}(M_n\geq cn)}{e^{\gamma cn}\mathbb{P}(M\geq cn)}=0.
			\end{align}
			\par
			\textbf{Case 2:~$c<\widehat{\mathbb{E}}_0[S_1]$.} Since $M_n\leq M$, we have
			\begin{align}\label{45ty6yhju2}
				e^{\gamma cn}\mathbb{P}(M\geq cn)&=e^{\gamma cn}\mathbb{P}(M\geq cn, M_n\geq cn)+e^{\gamma cn}\mathbb{P}(M\geq cn,M_n< cn)\cr
				&=e^{\gamma cn}\mathbb{P}(M_n\geq cn)+e^{\gamma cn}\mathbb{P}(M\geq cn,M_n< cn).
			\end{align}
			By the Markov property of the branching random walk, it follows that
			\begin{align}
				e^{\gamma cn}\mathbb{P}(M\geq cn,M_n< cn)&=e^{\gamma cn}\mathbb{E}(\mathbb{P}_{Z_n}(M\geq cn)
				\ind_{\{Z_{n}(\mathbb{R})>0, M_n< cn\}})\cr
				&\leq e^{\gamma cn}\mathbb{E}\left(
				\sum_{|\omega|=n}
				\mathbb{P}_{\delta_{V(\omega)}}(M>cn-1)\ind_{\{M_n< cn,Z_{n}(\mathbb{R})>0\}}\right)\cr
				&\leq e^{\gamma cn}\mathbb{E}\Bigg(\sum_{|\omega|=n}e^{-\gamma(cn-V(\omega)-1)}\ind_{\{M_n< cn\}}\Bigg)\cr
				&\leq e^{\gamma} e^{\gamma cn}\mathbb{E}\Bigg(\sum_{|\omega|=n}e^{-\gamma(cn-V(\omega))}\ind_{\{V(\omega)< cn\}}\Bigg),
			\end{align}
			where the second inequality follows from from \eqref{5tgt4htyh5y}.
			Combining \eqref{5hyth6y} and the many-to-one formula \eqref{many-to-one-formula},
			we conclude from the above inequality that
			\begin{align}
				e^{\gamma cn}\mathbb{P}(M\geq cn,M_n< cn)& \leq e^{\gamma}  m^n\mathbb{E}_0\left(e^{\gamma S_n}\ind_{\{S_n< cn\}}\right)\cr
				&=e^{\gamma}  \widehat{\mathbb{P}}_0(S_n<cn).
			\end{align}
			Since $c<\widehat{\mathbb{E}}_0[S_1]$, by the strong law of large numbers, we have $\lim_{n\to\infty}		\widehat{\mathbb{P}}_0
			(S_n<cn)=0.$ Thus,
			\begin{align}\label{566yu7ide3}
				\lim_{n\to\infty}e^{\gamma cn}\mathbb{P}(M\geq cn,M_n< cn)=0.
			\end{align}
			This, combined with \eqref{45ty6yhju2}, yields that
			\begin{align}
				\lim_{n\to\infty}e^{\gamma cn}\mathbb{P}(M_n\geq cn)-e^{\gamma cn}\mathbb{P}(M\geq cn)=0.
			\end{align}
			Since $M_n\leq M$, above yields that for $c\in(0,\widehat{\mathbb{E}}_0[S_1])$,
			\begin{align}\label{4ty6hy7hy6}
				\lim_{n\to\infty}\mathbb{P}(M_n\geq cn|M>cn)&=\lim_{n\to\infty}\frac{e^{\gamma cn}\mathbb{P}(M_n\geq cn)}{e^{\gamma cn}\mathbb{P}(M\geq cn)}\cr
				&=1+\lim_{n\to\infty}\frac{e^{\gamma cn}\mathbb{P}(M_n\geq cn)-e^{\gamma cn}\mathbb{P}(M\geq cn)}{e^{\gamma cn}\mathbb{P}(M\geq cn)}=1,
			\end{align}
			where in the last equality we used \eqref{4gh6tt5f}. This completes the proof of the theorem.
		\end{proof}
		\par
		Now we are going to prove Theorem \ref{theom4}. Recall that $M_n^{(0,\infty)}:=\max_{|\omega|=i,0\leq i\leq n, \forall u\preceq\omega , V(u)>0 }V(\omega)$ is the maximum of the branching random walk with killing up to time $n$. Without loss of generality, we only consider the non-lattice case.

		\begin{proof}[Proof of Theorem \ref{theom4}.] \textbf{Case 1:~$c>\widehat{\mathbb{E}}_0[S_1]$.}
			For $c>\widehat{\mathbb{E}}_0[S_1]$, similar to \eqref{step46}, using
			the fact that $M_n^{(0,\infty)}\leq M_n$, we have for $y>0$,
			\begin{align}
				&\lim_{n\to\infty}e^{\gamma cn}\mathbb{P}_{\delta_y}(M_n^{(0,\infty)}\geq cn) \leq \lim_{n\to\infty}e^{\gamma cn}\mathbb{P}_{\delta_y}(M_n\geq cn)= \lim_{n\to\infty}e^{\gamma cn}\mathbb{P}_{\delta_0}(y+M_n\geq cn)\cr
				&\leq \lim_{n\to\infty}
				e^{\gamma cn}\mathbb{P} \left(\max_{s\leq n} D_s^+ \geq
				\big((c-\widehat{\mathbb{E}}_0[S_1])n-y\big)e^{\gamma (cn-y)} \right)
				\leq  \lim_{n\to\infty} \frac{e^{\gamma y}}{(c-\widehat{\mathbb{E}}_0[S_1])n-y} \mathbb{E}\left(D_n^+\right)=0.\nonumber
			\end{align}
			where the last equality follows from the many-to-one formula \eqref{many-to-one-formula} and the ergodic theorem.
			Therefore, similar
			to \eqref{4gh6y6y55k}, above yields that for $c>\widehat{\mathbb{E}}_0[S_1]$,
			\begin{align}\label{4gh4rk}
				\lim_{n\to\infty}\mathbb{P}_{\delta_y}(M^{(0,\infty)}_n\geq cn|M^{(0,\infty)}>cn)=\lim_{n\to\infty}\frac{e^{\gamma cn}\mathbb{P}_{\delta_y}(M^{(0,\infty)}_n\geq cn)}{e^{\gamma cn}\mathbb{P}_{\delta_y}(M^{(0,\infty)}\geq cn)}=0.
			\end{align}
			\par
			\textbf{Case 2:~$c<\widehat{\mathbb{E}}_0[S_1]$.} Similar to \eqref{45ty6yhju2}, we have that
			\begin{align}\label{5hu7y5ggt5}
				e^{\gamma cn}\mathbb{P}_{\delta_y}(M^{(0,\infty)}\geq cn)=e^{\gamma cn}\mathbb{P}_{\delta_y}(M_n^{(0,\infty )}\geq cn)+e^{\gamma cn}\mathbb{P}_{\delta_y}(M^{(0,\infty)}\geq cn,M_n^{(0,\infty)}< cn).
			\end{align}
			According to the fact that $M^{(0,\infty)}\leq M$ and \eqref{5tgt4htyh5y}, we see that for $c< \widehat{\mathbb{E}}_0[S_1]$,
			\begin{align}
				e^{\gamma cn}\mathbb{P}_{\delta_y}(M^{(0,\infty)}\geq cn,M_n^{(0,\infty)}< cn) & \leq 	e^{\gamma }e^{\gamma cn} \mathbb{E}_{\delta_y}\left(\sum_{|\omega|=n} \ind_{\{V(v)>0, \forall v\preceq\omega\}} \ind_{\{V(\omega)<cn\}} e^{-\gamma(cn-V(\omega))} \right)\nonumber\\
				& =e^{\gamma (y+1)} \widehat{\mathbb{P}}_0\left(\underline{S}_n>-y, 
				S_n<cn-y \right)\cr
				&\leq e^{\gamma (y+1)}\widehat{\mathbb{P}}_0 
				(S_n<cn-y),
			\end{align}
			which tends to $0$ by the strong law of large numbers. By Theorem \ref{theom3}, we have for $c< \widehat{\mathbb{E}}_0[S_1]$,
			\begin{align}
				\lim_{n\to+\infty} e^{\gamma cn}\mathbb{P}_{\delta_y}(M_n^{(0,\infty )}\geq cn)- e^{\gamma cn}\mathbb{P}_{\delta_y}(M^{(0,\infty)}\geq cn)=0.
			\end{align}
			Similar to \eqref{4ty6hy7hy6}, above yields that 
			for
			$c\in(0,\widehat{\mathbb{E}}_0[S_1])$,
			\begin{align}
				\lim_{n\to\infty}\mathbb{P}_{\delta_y}(M^{(0,\infty)}_n\geq cn|M^{(0,\infty)}>cn)=1.
			\end{align}
			This completes the proof of the result.
			We are done.
		\end{proof}
		
		\medskip

		\textbf{Acknowledgements.}
		Haojie Hou is supported by the China Postdoctoral Science Foundation (No. 2024M764112).
		Shuxiong Zhang is the corresponding author and supported by the National Key R\&D Program (No. 2022YFA1006102).
		
		\vspace{2cm}

		\noindent{\bf Haojie Hou:}  School of Mathematics and Statistics, Beijing Institute of Technology, Beijing, China.
		
		\noindent{\bf Email:} {\texttt
			houhaojie@bit.edu.cn}
		
		\smallskip		
		\noindent{\bf Shuxiong Zhang:}  School of Mathematics and Statistics, Anhui Normal University, Wuhu, China.
		
		\noindent{\bf Email:} {\texttt
			shuxiong.zhang@mail.bnu.edu.cn}

	\end{document}